\documentclass[letterpaper]{article}
\usepackage{amsthm}
\usepackage{mathtools,mathrsfs,amssymb,latexsym,graphics,enumerate}
\usepackage[mathscr]{eucal}
\usepackage{amsmath,amsfonts,amsthm,amssymb}
\usepackage{color}
\usepackage{hyperref}
\usepackage[numbers,sort&compress]{natbib}
\usepackage[left=1in,right=1in,top=1in,bottom=1in]{geometry}

\numberwithin{equation}{section}

\newcommand{\rr}{\mathbb{R}}

\newcommand{\ga}{\Gamma}

\newcommand{\be}{\begin{eqnarray*}}
\newcommand{\bel}{\begin{eqnarray}}
\newcommand{\ee}{\end{eqnarray*}}
\newcommand{\eel}{\end{eqnarray}}
\newcommand{\ba}{\begin{aligned}}
\newcommand{\ea}{\end{aligned}}
\newcommand{\de}{\Delta}
\newcommand{\al}{\alpha}
\newcommand{\na}{\nabla}
\newcommand{\ep}{\epsilon}

\newcommand{\pa}{\partial}

\newcommand{\CC}{C_{1,\infty}}

\newtheorem{thm}{Theorem}

\newtheorem{cor}{Corollary}

\newtheorem{lem}{Lemma}

\newtheorem{rmk}{Remark}

\numberwithin{remark}{section}
\numberwithin{lem}{section}

\mathtoolsset{showonlyrefs=true}
\newcommand{\norm}[1]{\left\lVert#1\right\rVert}
\newcommand{\abs}[1]{\left\vert#1\right\vert}

\newcommand\Torus{{\mathbb T}}

\newcommand{\dss}{\displaystyle}
\newcommand{\cF}{E}

\title{On the $8\pi$-critical mass threshold of a Patlak-Keller-Segel-Navier-Stokes system}
\date{\today}
\author{Yishu Gong\footnote{\textit{
yishu.gong@duke.edu}, Department of Mathematics, Duke University}\quad 
Siming He\footnote{\textit{simhe@math.duke.edu}, Department of Mathematics, Duke University}}

\begin{document}
\maketitle
\abstract{In this paper, we proposed a coupled Patlak-Keller-Segel-Navier-Stokes system, which has dissipative free energy. On the plane $\rr^2$, if the total mass of the cells is strictly less than $8\pi$, classical solutions exist for any finite time, and their $H^s$-Sobolev norms are almost uniformly bounded in time. For the radially symmetric solutions, this $8\pi$-mass threshold is critical. On the torus $\mathbb{T}^2$, the solutions are uniformly bounded in time under the same mass constraint. 
}
\section{Introduction}
We consider the coupled Navier-Stokes-Patlak-Keller-Segel  equation modeling chemotaxis in moving fluid:
\begin{align}\label{pePKS-NS}
\left\{\begin{array}{rrrrr}\ba
\pa_t n+&u\cdot \na n+\na \cdot( n\na c)=\de n,\\
-\de& c=n,\\
\pa_t u+&(u\cdot \na )u+\na p=\de u +n\na c,\quad\na\cdot u=0,\\
n(t=&0,x)=n_{0}(x),\quad u(t=0,x)=u_{0}(x), \quad x\in  \rr^2.\ea\end{array}\right.
\end{align}
Here $n,c$ denote the cell density and the chemical density, respectively. The divergence-free vector field $u$ indicates the ambient fluid velocity. The first equation describes the time evolution of the cell density subject to chemotaxis-induced aggregation, diffusion caused by random Brownian motion, and transportation by ambient fluid flow $u$. Since the cells secrete the chemo-attractants, there exists a deterministic relation between the two distributions, $n$ and $c$.  The second equation specifies this connection. The assumption behind is that the chemo-attractant diffuses much faster than the fluid advection and cell aggregation,  and reaches equilibrium in a faster time-scale. The Newtonian potential is applied  to  determine $c$ uniquely, i.e., $c=-\frac{1}{2\pi}\log|\cdot|*n$. The third equation on the divergence-free vector field $u$ describes the fluid motion subject to forcing induced by the cells. The reasoning behind the coupling $n\na c$ is that in order to make the cells move without acceleration, the fluid exerts friction force on the moving cells, so reaction forces act on the fluid. The force $n \na c$ in the Navier-Stokes equation matches the aggregation nonlinearity in the cell density evolution. The same forcing appears in the {N}ernst-{P}lanck-{N}avier-{S}tokes system, see, e.g., \cite{ConstantinIgnatova19}.

If the ambient fluid velocity is identically equal to zero, i.e., $u\equiv 0$, the system \eqref{pePKS-NS} is the classical Patlak-Keller-Segel equation, which is first derived by C. Patlak \cite{Patlak}, and E. Keller and L. Segel \cite{KS}. The literature on the classical PKS model is extensive, and we refer the interested readers to the representative works, \cite{Biler06}, \cite{BlanchetEJDE06}, \cite{BlanchetCarrilloMasmoudi08}, \cite{Hortsmann} and the references therein. The classical PKS model preserves the total mass $M:=||n(t)||_1=||n_0||_1$ and is $L^1$ critical. If the initial data $n_0$ has total mass $M$ strictly less than $8\pi$, then the smooth solution exists for all time. Whereas if the initial data has total mass strictly larger than $8\pi$ and has a finite second moment, then the solution blows up in finite time, see, e.g., \cite{BlanchetEJDE06} and \cite{JagerLuckhaus92}. 

If the ambient fluid flow is not identically zero, i.e., $u\not\equiv 0$, the analysis of the long-time dynamics of the systems \eqref{pePKS-NS} are delicate. There is no heuristic arguments to rule out global solutions with large masses. Moreover, the underlying fluid flow might suppress the potential chemotactic blow-up in the system. This assertion is based on a series of work on the suppression of chemotactic blow-up through \emph{passive} fluid flows initiated by the work by A. Kiselev and X. Xu  \cite{KiselevXu15}. To simplify the analysis, in these models, the ambient fluid velocity fields $u$ are assumed to be independent of the time evolution of the cell densities. In these works, there are two main fluid machenisms to suppress the blow-up. The first mechanism is the fluid-mixing induced enhanced dissipation effect. The works in this direction are  \cite{BedrossianHe16}, \cite{He}, \cite{IyerXuZlatos}. The other mechanism to suppress the blow-up is the fast splitting scenario introduced in the paper  \cite{HeTadmor172}.
        
{The model \eqref{pePKS-NS} takes into account the \emph{active} chemotaxis-fluid interaction.  
The literature concerning coupled chemotaxis-fluid systems is vast. We refer the interested readers to the papers \cite{Lorz10}, \cite{Lorz12}, \cite{LiuLorz11}, \cite{DuanLorzMarkowich10},  \cite{FrancescoLorzMarkowich10}, \cite{Tuval05}, \cite{Winkler12}, \cite{TaoWinkler}, \cite{ChaeKangLee13}, \cite{KozonoMiuraSugiyama}, \cite{Winkler14, Winkler16, Winkler17, Winkler20, Winkler21} and the references therein. A lot of works are devoted to the study of parabolic-parabolic Patlak-Keller-Segel equations subject to active fluid motions. The coupling between the chemotaxis and the fluid in these models is through the gravity-buoyancy relation. The closest models to ours are proposed by A. Lorz \cite{Lorz12} and H. Kozono et al. l \cite{KozonoMiuraSugiyama}. The chemical densities $c$ in these models are also determined through elliptic-type equations. On the other hand, these models consider buoyancy forcing instead of the reaction force from the cells.
} 

{Another biologically relevant coupled Patlak-Keller-Segel-Navier-Stokes models was introduced by I. Tuval et al. \cite{Tuval05},
\begin{align}
\left\{\begin{array}{rrrrr}\ba
\pa_t n+&u\cdot \na n+\na \cdot( n\na c)=\de n,\\
\pa_t c+&u\cdot \na c=\de c-n f(c),\\
\pa_t u+&(u\cdot \na )u+\na p=\de u +n\na \phi,\quad\na\cdot u=0.\ea\end{array}\right.
\end{align}
Here the chemicals (oxygen) are transported by the fluid stream $u$ and get consumed with rate $f(c)>0$. Due to buoyancy, the cells exert force $n\na \phi$ on the fluid.  Since the chemicals are consumed along the dynamics, one expects that the cell density will not concentrate to form finite-time singularities.  However, the parabolic nature of the chemical evolution makes the analysis  challenging. In the papers \cite{Winkler14}, \cite{Winkler16}, \cite{Winkler17}, \cite{Winkler21}, global regularity, long time behaviors and Leray structure of the system are explored in detail. }

In this paper, we study the 
{critical} mass threshold, below which, the solutions of the system \eqref{pePKS-NS} are guaranteed to exist for all finite time. The main advantage of the proposed model \eqref{pePKS-NS} is that it possesses a naturally decreasing free energy, 
\begin{align}
E[n,u]:=&\int_{\rr^2} n\log n -\frac{1}{2}nc +\frac
{1}{2} |u|^2dx.\label{Free_energy_PKSNS}
\end{align} 
Moreover, since the vector field $u$ is divergence-free, the density equation for $n$ possesses a divergence structure and hence preserves the $L^1$ norm. 

On the whole plane, we prove the following theorem.
\begin{thm}[{Plane $\rr^2$ case}]\label{thm:R2}
Consider solutions $(n,u)$ to the equation \eqref{pePKS-NS} subject to  initial conditions $(n_0, u_0)\in H^s(\mathbb{R}^2)\times (H^s(\mathbb{R}^2))^2,\, s\geq3$ and $n_0(1+|x|^2)\in L^1(\rr^2)$. If the initial mass is strictly less than $8\pi$, 
\begin{align}
M:=||n_0||&_{L^1(\mathbb{R}^2)}<8\pi,
\end{align} 
then there exists a constant $C$, which depends on the initial data, such that the following estimate holds 
\begin{align}
||n(t)||_{H^s}+||u(t)||_{H^s}\leq C(n_0,u_0,\delta)e^{\delta t},\quad \forall t\in[0,\infty), \label{Exponential_H_s}
\end{align}
where $0<\delta$ is an arbitrary small constant. Therefore, the strong solutions $(n, u)$ exist on arbitrary finite time interval $[0,T],\, \forall T<\infty$. 
\end{thm}
\begin{rmk}To our knowledge, this is the first critical-mass result in the coupled Patlak-Keller-Segel-Navier-Stokes systems. 
\end{rmk}
\begin{rmk}
The exponential bounds stated in the theorem might not be  optimal. We conjecture that the solutions subject to subcritical mass are uniformly bounded in time. 

In the radially symmetric setting, the long-time behavior of the solutions are better understood. We will show that the chemotactic blow-up occurs if the initial density $n_0$ has total mass $||n_0||_{L^1}>8\pi$ and has  finite second moment (Corollary \ref{cor:radially_symmetric_case}). On the other hand, if the total mass is strictly less than $8\pi$, and the initial second moment is finite, then the $L^2$-norm of the  solutions $(n, \mathrm{curl}u)$ decay to zero as time approaches infinity with algebraic  rate (Theorem \ref{thm:radial_long_time}).
\end{rmk}
\begin{rmk}
Extending Theorem \ref{thm:R2}, which concerns the parabolic-elliptic Patlak-Keller-Segel-Navier-Stokes system, to the fully parabolic setting is both interesting and challenging.
\end{rmk}
One of the main obstacles to uniform in time bounds on the solutions is the lack of control over the second moment. To properly illustrate that this is the only obstacle, we choose to study the model \eqref{pePKS-NS} on torus $\mathbb{T}^2$, and show that under the same subcritical mass constraint, the solutions are uniformly bounded in time. To this end, due to compatibility with the boundary conditions involved, we have to adjust the equation \eqref{pePKS-NS} accordingly. Here we specified the equation on the torus $\mathbb{T}^2$:
\begin{align}\label{pePKS-NS-Torus}
\left\{\begin{array}{rrrrr}\ba
\pa_t n+&u\cdot \na n+\na \cdot(n\na c )=\de n,\\
-\de& c=n-\overline{n},\quad \overline{n}=\frac{1}{|\mathbb{T}^2|}\int_{\mathbb{T}^2}ndx\\
\pa_t u+&(u\cdot \na )u+\na p=\de u +n\na c,\quad\na\cdot u=0,\\
n(t=&0,x)=n_{0}(x),\quad u(t=0,x)=u_{0}(x), \quad x\in \mathbb{T}^2.\ea\end{array}\right.
\end{align}
Without loss of generality, we assume that the size of the torus is $|\mathbb{T}|=1$. The chemical $c$ is determined by $\dss c(x)=-\int_{\Torus^2} B_{\Torus^2}(x,y)n(y)dy$, where $B_{\Torus^2}(x,y)$ is the Green's function of the Laplacian $\de$ on the torus $\Torus^2$.

The second main theorem of the paper describes the 
{global well-posedness} of the equations \eqref{pePKS-NS-Torus}. 
\begin{thm}[Torus $\mathbb{T}^2$ case]\label{thm:Torus}
Consider the solution to the equation \eqref{pePKS-NS-Torus} subject to $H^s$ initial data $(n_0, u_0)\in H^s(\mathbb{T}^2)\times (H^s(\mathbb{T}^2))^2,\, s\geq 3$. If the initial mass $M:=||n_0||_{L^1(\mathbb{T}^2)}$ is strictly less than $8\pi$, i.e., $M<8\pi$, then the solution $(n, u)$ has uniform-in-time bounded $H^s$ Sobolev norm, i.e.,
\begin{align}
||n||_{L_t^\infty([0,\infty);H^s)}+||u||_{L_t^\infty([0,\infty);H^s)}\leq C_{H^s}(||n_0||_{H^s},||u_0||_{H^s})<\infty.
\end{align}  
\end{thm} 
\begin{rmk}
We comment that similar uniform-in-time bounds are obtained in the parabolic-parabolic setting given that the total mass is small enough, \cite{Winkler20}.
\end{rmk}
\subsection{Ideas of the Proof}
We discuss the idea behind Theorem \ref{thm:R2}. Recall the free energy $E$ for the system \eqref{pePKS-NS} and the second moment $V$
\begin{align}
V[n]:=&\int_{\rr^2} n|x|^2dx.\label{Second_moment}
\end{align} 
The existence of a decreasing free energy is crucial in obtaining sharp critical mass results in Patlak-Keller-Segel type equations. We recall that for the classical PKS equation ($u\equiv 0$), there exists a dissipative free energy, 
\begin{align*}
E_{\mathrm{classic}}=\int_{\rr^2} n\log n -\frac{1}{2}nc dx.
\end{align*} 
However, if  the fluid transport structure is introduced in the cell density evolution equation, the classical free energy will no longer decay in general. This is one of the main difficulties in analysing the coupled Patlak-Keller-Segel-Navier-Stokes systems. However, our coupled system \eqref{pePKS-NS} possesses a new dissipative free energy \eqref{Free_energy_PKSNS}. This is the main content of the next lemma.

\begin{lem}\label{Lem:Free_Energy_Decay}
Consider regular solutions $(n,u)$ to the equation \eqref{pePKS-NS}. Further assume that 
{$(n,u)\in \mathrm{Lip}_t([0,T];H_x^s(\rr^2)\times (H_x^s(\rr^2))^2)$,\, $s\geq 3$}, and $n(1+|x|^2)\in L_t^\infty([0,T];L_x^1(\rr^2))$. Then the free energy \eqref{Free_energy_PKSNS} is dissipated along the dynamics \eqref{pePKS-NS}, i.e.,
\begin{align}\label{Free_Energy_dissipation}
E[n(t), u(t)]= E[n_0,u_0]-\int_0^t\int_{\rr^2} n|\na \log n-\na c|^2 dxds-\int_0^t\int_{\rr^2}|\na u|^2dxds, \quad \forall t\in[0,T]. 
\end{align}
\end{lem}
\begin{proof}
Direct calculation using integration by parts and divergence free condition of $u$ yields that
\begin{align*}
\frac{d}{dt}E
=&-\int_{\rr^2} n|\na \log n-\na c|^2 dx-\int_{\rr^2} n u\cdot \na c dx\\&-\int_{\rr^2}|\na u|^2dx-\int_{\rr^2}u\cdot ((u\cdot \na)u) dx-\int_{\rr^2} u\cdot\na p dx+\int _{\rr^2}n u\cdot\na cdx \\
=&-\int_{\rr^2} n|\na \log n-\na c|^2 dx-\int_{\rr^2}|\na u|^2dx\leq 0.
\end{align*} Here in the last line, we apply the relation that \[\int_{\rr^2} u\cdot ((u\cdot \nabla)u)dx=\int_{\rr^2} u \cdot \nabla\left(\frac{|u|^2}{2}\right)dx=0.\] 
Now integration in time yields the equation  \eqref{Free_Energy_dissipation}. 
\end{proof}
Before utilizing the dissipative free energy to derive global well-posedness of the solutions, we present the following local-wellposedness result, whose proof will be postponed to the appendix.
\begin{thm}\label{thm:local_well_posedness}[Local well-posedness] 
Consider the solutions to the equation \eqref{pePKS-NS} subject to $H^s$ initial data, i.e., $(n_0, u_0)\in H^s(\mathbb{R}^2)\times (H^s(\mathbb{R}^2))^2,\, s\geq3$. 
There exists a small constant $\ep=\ep(||n_0||_{L^1\cap H^1}, ||u_0||_{H^1})$  such that the Sobolev $H^s$ norms of the solutions are bounded on the time interval $[0,\ep]$
\begin{align}
||n(t)||_{H^s}+ ||u(t)||_{H^s}<\infty,\quad \forall t\in[0, \ep]. 
\end{align}
\end{thm}
Next we recall from the classical PKS literature that, to propagate higher regularities of solutions, the entropy bound of the solution is essential, see, e.g., \cite{BlanchetEJDE06}, \cite{BlanchetCarrilloMasmoudi08}. We present here a similar criteria which guarantees propagation of regularity. 
\begin{thm}\label{thm:S+n_condition}
Consider solution $(n,u)$ to the equation \eqref{pePKS-NS} subject to initial conditions $(n_0, u_0)\in H^s(\mathbb{R}^2)\times (H^s(\mathbb{R}^2))^2,\, s\geq3, \, n_0(1+|x|^2)\in L^1(\rr^2)$. If the positive part of the entropy is bounded, i.e., 
\begin{align}\label{S+n_condition}
S^+[n(t)]:=\int_{\rr^2} n(t,x)\log ^+ n(t,x)dx\leq C_{L\log L}<\infty,\quad\forall t\in [0,T],
\end{align}
and the energy of the fluid $u$ is bounded, i.e.,
\begin{align}\label{Energy_u_condition}
||u(t)||_2^2\leq C_{u;L^2}^2<\infty, \quad\forall t\in[0,T],
\end{align}
then the solution has bounded $H^s,\, s\geq 3$ norms on the same time interval 
 \begin{align}
 ||n(t)||_{H^s}+||u(t)||_{H^s}\leq C_{H^s}(C_{L\log L},C_{u;L^2}, ||n_0||_{H^s},||u_0||_{H^s})<\infty, \quad \forall t\in [0,T].
 \end{align}
\end{thm}  
We recall the standard procedure to check the criterion \eqref{S+n_condition} for the classical PKS equations. In the subcritical regime, i.e., $||n_0||_1<8\pi$, combining the decaying free energy \eqref{Free_Energy_dissipation} and the logarithmic-Hardy-Littlewood-Sobolev inequality \eqref{log HLS}  yields the uniform-in-time bound on the entropy
\begin{align}
\sup_t S[n(t)]:=&\sup_t\int_{\rr^2} n(t,x)\log n(t,x)dx=\sup_t\left(\int_{\rr^2} n(t,x)\log^+ n(t,x)dx-\int_{\rr^2}n(t,x)\log ^-n(t,x) dx\right)\\
=:& \sup_{t}(S^+[n(t)]-S^-[n(t)])<\infty.\label{Entropy}
\end{align} Here $\log^+,\, \log^-$ denote the positive part and the negative part of the logarithmic function, respectively. As a result, we observe that as long as the negative component of the entropy $S^-[n]$ is bounded, then criterion \eqref{S+n_condition} is checked. It is classical to apply the second moment $V$ \eqref{Second_moment} bound to estimate the negative part of the entropy $S^-[n]$ (see, e.g., inequality \eqref{S-n_control}). We summarize the above heuristics in the next theorem, with our system in consideration.  
\begin{thm}\label{thm:Second_moment_criterion}
Consider solutions $(n, u)$ to \eqref{pePKS-NS}  on the time interval $[0,T]$, subject to initial conditions $(n_0,u_0)\in (H^s(\rr^2),(H^s(\rr^2))^2), \, s\geq 3$, $n_0(1+|x|^2)\in L^1(\rr^2)$. If the initial mass is strictly less than $8\pi$, 
\begin{align}
M:=||n_0||&_{L^1(\mathbb{R}^2)}<8\pi,
\end{align} 
and the second moment is bounded on the time interval $[0,T]$, 
\begin{align}
V[n(t)]
\leq C_V<\infty,\quad\forall t\in [0,T],\label{Second_Moment_Bound}
\end{align}  
then the entropy bound \eqref{S+n_condition} and the energy bound \eqref{Energy_u_condition} hold, i.e.,
\begin{align}
\int_{\rr^2}n(t,x)\log^+n(t,x) dx+||u(t)||_2^2\leq C(C_V,M, E[n_0,u_0])<\infty, \quad \forall t\in[0,T].
\end{align} 
\end{thm}
The condition \eqref{Second_Moment_Bound} can be easily checked for the following two cases: a) solutions on the bounded domain $\mathbb{T}^2$ (Theorem \ref{thm:Torus}); b) radially symmetric solutions on $\rr^2$: 
\begin{cor}[{Plane $\rr^2$, Radially symmetric solutions}]\label{cor:radially_symmetric_case}
Consider the equation \eqref{pePKS-NS} subject to $H^s$ radially symmetric initial data $(n_0, u_0)\in H^s(\mathbb{R}^2)\times (H^s(\mathbb{R}^2))^2,\, s\geq 3$. 
{Further assume that the second moment is finite $\dss \int_{\rr^2}n_0|x|^2dx<\infty$.} If the initial mass $M:=||n_0||_{L^1(\mathbb{R}^2)}$ is strictly less than $8\pi$, i.e., $M<8\pi$, then the solution $(n, u)$ has bounded $H^s$ Sobolev norm for any finite time $t<\infty$. 
{On the other hand, if the total mass of the initial density $n_0$ is greater than $8\pi$, i.e., $||n_0||_{L^1(\rr^2)}>8\pi$, then the solution $(n, u)$ blows up in finite time.}
\end{cor}
However, it is difficult to apply Theorem \ref{thm:Second_moment_criterion} to general solutions to \eqref{pePKS-NS} on the plane $\rr^2 $, since controlling second moment \eqref{Second_Moment_Bound} requires $||u||_\infty$ information, which is typically missing in the a-priori estimates. Here we develop a new method to check criterion \eqref{S+n_condition}:
 
We modify the free energy $E$ \eqref{Free_energy_PKSNS} so that the new negative component of the entropy $S^-[n]$ is bounded in terms of  the $L^1$ norm of the density $n$. As a result, there is no need for the second moment control. To this end, we replace the logarithmic function by its degree two Taylor approximation when the argument $n$ is smaller than designated threshold. The drawback is that the modified free energy can potentially grow slowly. However, this is enough to derive the $S^+[n]$ bound for any finite time. As a result, we end up with the exponential bounds with arbitrarily small growth rate in the $H^s$ Sobolev norms. Uniform-in-time bounds on the solutions are still open.  Details of this modified free energy can be found in Section \ref{Sec:R2}.
\begin{thm}\label{thm:modified_free_energy_result}
Consider regular solutions to the equation \eqref{pePKS-NS}, 
{subject to initial conditions $(n_0, u_0)\in H^s(\rr^2)\times( H^s(\rr^2))^2,\, s\geq 3, \, n_0(1+|x|^2)\in L^1(\rr^2)$}. If the initial mass is strictly less than $8\pi$, 
\begin{align}
M:=||n_0||&_{L^1(\mathbb{R}^2)}<8\pi,
\end{align} 
then the entropy bound \eqref{S+n_condition} and the energy bound \eqref{Energy_u_condition} hold on any finite time interval $[0,T]\subset[0,\infty)$. Moreover, for any small constant $\delta>0$, there exists a constant $C(E[n_0,u_0],M,\delta)$ such that
\begin{align}\label{Linear_bound_on_S_+}
S^+[n(t)]+||u(t)||_2^2\leq &C(E[n_0,u_0],M,\delta)+\delta t,\quad \forall t\in [0,\infty).
\end{align} 
\end{thm}  
From the linearly growing bound on the positive component of  the entropy $S^+[n]$ and the energy $||u(t)||_2^2$, one can derive the exponential-in-time bound on the $H^s$-Sobolev norms \eqref{Exponential_H_s} through standard energy estimates. This concludes the proof of Theorem \ref{thm:R2}. 

In general, the long time asymptotic behavior of the solution to \eqref{pePKS-NS} is not clear. However, for radially symmetric solutions, we have the following description.
\begin{thm}\label{thm:radial_long_time}
Consider {radially symmetric solutions} to the equation \eqref{pePKS-NS}  subject to the subcritical mass constraint $||n_0||_1<8\pi$ and the  conditions in Corollary  \ref{cor:radially_symmetric_case}.  The $L^2$-norms of the solutions undergoes polynomial decay in the sense that
\begin{align}\label{Polynomial_decay}
||n(t)||_{L^2}^2+||\mathrm{curl}u(t)||_{L^2}^2\leq \frac{C}{1+2t},\quad \forall t\in[0,\infty),
\end{align}
where $C$ is a constant depending on the initial data. 
\end{thm} 
\begin{rmk}
By applying the same argument as in the proof of Theorem \ref{thm:S+n_condition}, we obtain that the $H^s$ norms of the solutions are uniformly bounded in time. 
\end{rmk}

The paper is organized as follows: In Section \ref{Sec:R2}, we treat the planar case and prove Theorem \ref{thm:R2}, Theorem \ref{thm:S+n_condition}, Theorem \ref{thm:Second_moment_criterion}, Corollary \ref{cor:radially_symmetric_case} and Theorem \ref{thm:radial_long_time}. In Section \ref{Sec:T2}, we treat the torus case and prove Theorem \ref{thm:Torus}.

\noindent
\textbf{Notation:} Throughout the paper, the constants $B,C$ are changing from line to line. However, the constants $C_{(\cdot)}$, e.g., $C_{L^2},\, C_{L\log L}$ will be defined and fixed unless otherwise stated. An exception of this rule is the constants $C_{GNS}$ and $C_N$, they are the constants appeared in the Gagliardo-Nirenberg-Sobolev inequalities and the Nash inequalities and are changing from line to line. 


We denote  $\mathbb{P}$ as the Leray projection, i.e.,
\begin{align}\label{Leray_Projection}
\mathbb{P}u=u-\na \de^{-1}(\na \cdot u).
\end{align}
Here the operator should be understood as the pseudo-differential operators. Explicitly speaking, for vector field $u=(u^1,u^2)$, we have
\begin{align}
\mathbb{P}u^i(x)=\left(\sum_{j=1}^2 \bigg(\delta_j^{i}-\frac{k_i k_j}{|k|^2}\bigg)\widehat{u^j}(k)\right)^\vee, \quad i \in\{1,2\},
\end{align}
where $\widehat{(\cdot)}$ and the $(\cdot)^\vee$ denote the Fourier transform and inverse transform on the plane $\rr^2$ or the torus $\mathbb{T}^2$ respectively, and the $\delta_j^i$ is the Kronecker delta function. Further properties of the Leray projection include that it is a self-adjoint Fourier multiplier and it is a continuous map from $L^2$ to $L^2$.  
Now we define the Stokes operator as $\mathbb{P}(-\de)$. Furthermore, we define the bilinear form
\begin{align}
B(u,v)=\mathbb{P}((u\cdot \na) v).
\end{align}
Properties of these operators can be found in classical literature, e.g., Chapter 2 of \cite{KuksinShirikyan12}.

The following multi-index notation is adopted:
\begin{align}
\pa_x^\al =\pa_{x_1}^{\al_1}\pa_{x_2}^{\al_2},\quad |\al|=|\al_1|+|\al_2|.
\end{align} 
Moreover, we denote $\beta<\al$ if $\beta_1\leq\al_1$, $\beta_2\leq\al_2$, and at least one of the inequalities is strict. 

Recall the classical $L^p$ norms and Sobolev $H^s$ norms:
\begin{align}
||f||_{L_x^p}=&||f||_p=\left(\int |f|^p dx\right)^{1/p};\quad ||f||_{L_t^q([0,T]; L^p_x)}=\left(\int_0^T||f(t,x)||_{L_x^p}^qdt\right)^{1/q};\\
||f||_{H_x^s}=&\left(\sum_{|\al|\leq s}||\pa_x^\al f||_{L_x^2}^2\right)^{1/2};\quad
||f||_{\dot H_x^s}=\left(\sum_{|\al|=s}||\pa_x^\al f||_{L_x^2}^2\right)^{1/2};\quad||\na^i f||_{L^2}=\left(\sum_{|\al|=i}||\pa_x^\al f||_{L^2}^2\right)^{1/2}.
\end{align}
\ifx 
By Lemma 23 on 
https://terrytao.wordpress.com/2009/04/30/245c-notes-4-sobolev-spaces/
, we have that $C_c^\infty$ functions are dense in the space $H^s(\rr^n)$. As a result, we can approximate the $H^s$ function by $C_c^\infty$ function and make the difference approaches zero. There should be no problem in integration by parts if the function is $H^s(\rr^2),\, s\geq 3$. A detailed argument is as follows: Let $f,g\in H^1(\rr^2)$, $\widetilde{f}, \widetilde{ g}\in C_{c}^\infty(\rr^2)$ such that $||f-\widetilde{f}||_{H^1}\leq \ep,\quad ||g-\widetilde{g}||_{H^1}\leq \ep$. Now we have that
\begin{align}
\int f\pa_x g dx &\leq \int \widetilde{f}\pa_x \widetilde{g}dx+|\int f\pa_x (g-\widetilde{g})dx+\int (f-\widetilde{f})\pa_x(\widetilde{g})dx|\leq \int\widetilde{f}\pa_x \widetilde{g}dx+||f||_{L^2}||g-\widetilde{g}||_{H^1}+||g||_{H^1}||f-\widetilde {f}||_{L^2}\\
&\leq \int \widetilde{f}\pa_x \widetilde{g}dx +C(f,g)\ep=-\int \pa_x \widetilde{f} \widetilde{g}dx+ C(f,g)\ep\\
&\leq- \int \pa_x f g dx+||\widetilde f||_{H^1}||g-\widetilde{g}||_2+||f-\widetilde{f}||_{H^1}|| g||_2+C(f,g)\ep\leq -\int \pa_x f g dx+2C(f,g )\ep.
\end{align}
Since $\ep$ is arbitrary, we have that
\begin{align}
\int f\pa_x g dx\leq -\int \pa_x fg dx. 
\end{align}
Similar approximation argument yields that
\begin{align}
\int f\pa_x gdx\geq -\int \pa_x fg dx.
\end{align}
To conclude, we have that 
\begin{align*}
\int f\pa_x g dx=-\int \pa_xf gdx,\quad f,g\in H^1(\rr^2).
\end{align*}\fi
\section{Planar Case: $\mathbb{R}^2$}\label{Sec:R2}
The section is organized as follows. We first prove  Theorem \ref{thm:S+n_condition}. The proof will serve as a prototype for our later analysis on the torus $\Torus^2$. Next we prove Theorem \ref{thm:Second_moment_criterion}, which assumes that the cell density $n$ has bounded second moment on the time interval $[0,T]$. Then we prove Corollary \ref{cor:radially_symmetric_case} by showing that the second moment bound \eqref{Second_Moment_Bound} is checked in the radially symmetric setting. Finally, we introduce the modified free energy to prove Theorem \ref{thm:R2}.
 
\begin{proof}[Proof of Theorem \ref{thm:S+n_condition}]
In the proof, we focus on deriving the a-priori estimates for the $H^s,\, s\geq 3$ Sobolev norms of the solutions $(n,u)$. Then by a standard limiting procedure and contraction mapping argument, one can deduce the existence and uniqueness of the solutions to the equation \eqref{pePKS-NS}. The proof is decomposed into steps. 

\noindent
\textbf{Step \# 1: $L^p$ estimate of the density $n$.}
First we recall that due to the divergence structure of the cell density equation in  \eqref{pePKS-NS}, the total mass of the cells are conserved along the dynamics. Therefore, we set $M:=||n(t)||_1=||n_0||_1$. In order to estimate the $L^p,\, p>1$ norm of the density $n$, we decompose it as follows:
\begin{align} n=(n-K)_++\min \{n, K\},\quad K>1.
\end{align}
Since $\min\{n,K\}$ has bounded $L^p$ norm, it is enough to estimate the size of $(n-K)_+$. To this end, define the following quantity:
\begin{align}
\eta_K:=\int_{\rr^2}(n-K)_+ dx.
\end{align}
Since the positive part of the entropy is bounded on the interval $[0,T]$ \eqref{S+n_condition}, direct estimation yields that
\begin{align}\label{small_eta_K}
\eta_K\leq \int_{\rr^2} (n-K)_+\frac{\log^+ n}{\log K}dx\leq \frac{C_{L\log L}}{\log K}.
\end{align} 
As a result, if we choose the vertical cut-off level $K$ large enough, the $\eta_K$ can be made arbitrarily small. Next we combine the smallness of $\eta_K$ \eqref{small_eta_K}, the divergence free condition of the fluid vector field $u$, the Gagliardo-Nirenberg-Sobolev inequality and the Nash inequality to estimate the time evolution of the $L^{2}$ norm of the truncated density $(n-K)_+$ as follows:
\begin{align}
\frac{1}{2}\frac{d}{dt}||(n-K)_+||_{2}^{2}\leq&-\int |\na (n-K)_+|^2 dx +\frac{1}{2}\int (n-K)_+^{3}dx
+\frac{3}{2}K\int (n-K)_+^2 dx+K^2M\\
\leq&-(1-C_{GNS}\eta_K)||\na (n-K)_+||_{2}^2+2K||(n-K)_+||_{2}^{2}+K^2M\\
\leq&-\frac{1}{2}||\na (n-K)_+||_{2}^2+2K||(n-K)_+||_{2}^{2}+K^2M\\
\leq &-\frac{1}{2C_N M^2}||(n-K)_+||_2^4+2K||(n-K)_+||_2^2+K^2M.\label{Proof_of_n_L_2_0}
\end{align}
As a result, we see that \begin{align}
||n(t)||_{2}\leq ||(n(t)-K)_+||_2+||\min\{n(t),K\}||_2\leq C(||n_0||_2,C_{N}, M, K)+K^{1/2}M^{1/2},\quad \forall t\in [0,T].
\end{align}
Since in the estimation above, we choose $K$ such that 
\begin{align}
\frac{C_{L\log L}}{\log K}\leq \frac{1}{2C_{GNS}},
\end{align}
we have that $K$ can be any constant greater than $\exp\{2C_{GNS}C_{L\log L}\}.$ To conclude, we have that
\begin{align}\label{Proof_of_n_L2}
||n(t)||_2\leq C_{L^2}(||n_0||_{2},M,C_{L\log L})<\infty,\quad \forall t\in[0,T].
\end{align}
  
Direct estimation of the time evolution of the $L^4$ norm of the cell density $n$ with the $L^2$ bound on the cell density  $n$ \eqref{Proof_of_n_L2}, the Gagliardo-Nirenberg-Sobolev equality, and the Nash inequality yields,
\begin{align}
\frac{1}{4}\frac{d}{dt}||n||_{4}^4 \leq& -\frac{3}{4}||\na (n^2)||_{{2}}^2+\frac{3}{4}||n^2||_{{5/2}}^{5/2}\\
\leq&-\frac{3}{4}||\na (n^2)||_{2}^2+C_{GNS}||\na (n^2)||_{2}^{1/2}||n^2||_{2}^2\\
\leq& -\frac{||n^2||_{2}^4}{C_{N}||n^2||_{1}^2}+C_{GNS}||n^2||_{2}^{8/3}\\
\leq&-\frac{||n||_{4}^8}{C_{N}C_{L^2}^4}+C_{GNS}||n||_{4}^{16/3}.
\end{align}Therefore we obtain that 
\begin{align}
||n(t)||_{4}\leq C_{L^4}(||n_0||_4,C_{L^2}(||n_0||_2, M, C_{L\log L}))<&\infty,\quad \forall t\in[0,T].
\end{align}

Combining the Morrey's inequality, the Calderon-Zygmund inequality, and the $L^p$ bounds of the density $n$ \eqref{C2infty} yields that 
\begin{align}
||\na c(t)||_{L^\infty(\mathbb{R}^2)}\leq& C\norm{n(t)}_{L^3(\rr^2)}\leq C_{\na c;\infty}(C_{L^4}, M)<\infty,\quad \forall t\in [0,T].\label{na_c_L_infty}
\end{align}
Since the vector field $u$ is divergence free, the fluid transport term $u\cdot \na n$ has no impact on the direct $L^p$ energy estimate on the cell density $n$. Now by the standard Moser-Alikakos iteration, we have that there exists a finite constant $\CC$ such that the $L^p$ norms are bounded as follows
\begin{align}\label{C2infty}
||n(t)||_{L^1\cap L^\infty}\leq C_{1,\infty}(||n_0||_{L^1\cap L^\infty},C_{L\log L})<\infty,\quad \forall t\in[0,T].
\end{align}For the iteration argument in the classical Patlak-Keller-Segel equation setting, we refer the readers to the Lemma 3.2 in \cite{CalvezCarrillo06} or the paper \cite{Kowalczyk05}. For the Patlak-Keller-Segel equation subject to ambient divergence free vector fields, we refer to the appendix of \cite{KiselevXu15}.

\noindent
\textbf{Step \# 2: $H^s$ estimate of the density $n$ and the velocity $u$.} Before estimating the $\dot H^1$ norms of the solutions $(n,u)$, we present two estimates on the chemical gradient $\na c$. Combining the $L^p$ boundedness of the Riesz transform for $p\in(1,\infty)$ on $\rr^2$ and the $L^p$-bounds of the density $n$ \eqref{C2infty} yields that
\begin{align}
||\na^2 c||_2=||\na^2(-\de)n||_2\leq C ||n||_2\leq C\CC,\quad
||\na^2 c||_4=||\na^2(-\de)n||_4\leq C ||n||_4\leq C\CC.\label{na2c_R2}
\end{align} 
After these preparation, we first estimate the $\dot H^1$ norm of the velocity fields $u$. We apply the Leray projection $\mathbb{P}$ \eqref{Leray_Projection} on the fluid equation \eqref{pePKS-NS} to eliminate the pressure term and end up with the following, 
\begin{equation}\label{eqn:flow}
    \partial_t u+B(u,u)=\Delta u+\mathbb{P}(n\nabla c),\quad B(u,u):=\mathbb{P}((u\cdot \nabla)u).
\end{equation}
Here we use the fact that $\mathbb{P}u=u$ since $u$ is divergence free. Moreover, since the symbol of $\mathbb{P}$ is bounded, the projection $\mathbb{P}$ maps $L^2$ space to $L^2$ space.  We also recall the classical identity: for divergence-free $u\in L^2\cap H^2$, 
\begin{align}
\int B(u,u)\cdot\de udx=0.\label{B_identity}
\end{align}
The proof of the identity, which involves the stream function of $u$, can be found in \cite{KuksinShirikyan12} Lemma 2.1.16. The $H^2$-regularity required by this equality is guaranteed by the local well-posedness theorem \ref{thm:local_well_posedness}. 
Now we estimate the time evolution of the $\dot H^1$ seminorm of the velocity $u$ with the equality \eqref{B_identity}, the  divergence-free condition of $u$, the self-adjoint property of $\mathbb{P}$,  the Gagliardo-Nirenberg-Sobolev inequality, the chemical gradient estimates \eqref{na_c_L_infty}, \eqref{na2c_R2}, and the $L^p$ controls of the cell density $n $ \eqref{C2infty} as follows:
\begin{align*}
\frac{1}{2}\frac{d}{dt}\sum_{j=1}^2||\pa_{x_j} u||_2^2=&-\sum_{j=1}^2\sum_{k=1}^2\int |\pa_{x_k} \pa_{x_j} u|^2 dx-\sum_{j=1}^2\int \pa_{x_j} B(u,u) \cdot \pa_{x_j} u dx+\sum_{j=1}^2\int \pa_{x_j} \mathbb{P}( n\na c)\cdot \pa_{x_j} u dx\\
\leq &-\frac{1}{2}||\na^2 u||_2^2+C ||\na^2 u||_2||  n||_2||\na c||_\infty \\
\leq&-\frac{||\na u||_2^4}{2C_{GNS}||u||_2^2}+C|| n||_2^2C_{1,\infty}^2 .
\end{align*}
As a result, we recall the assumption \eqref{Energy_u_condition} and obtain  that
\begin{align}
||\na u(t)||_{L_x^2}\leq C_{u;H^1}(C_{u;L^2},||\na u_0||_2, ||n_0||_{L^1\cap L^\infty}),\quad \forall t\in[0,T].\label{na_u_L2}
\end{align}
Similarly, we estimate the time evolution of the $\dot H^1$ seminorm of $n$ using the divergence free property of $u$,  the Gagliardo-Nirenberg-Sobolev inequality, the chemical gradient estimate \eqref{na_c_L_infty}, \eqref{na2c_R2}, the $\na u$ bound \eqref{na_u_L2}, and the $L^2$ bound of the density $n$ \eqref{Proof_of_n_L2} as follows:
\begin{align*}
\frac{1}{2}\frac{d}{dt}||\na n||_2^2\leq&-\frac{1}{2}{||\na ^2 n||_2^2}+||\na n||_4^2||\na u||_2+||\na^2 n||_2||\na n||_2||\na c||_\infty+||\na^2n||_2|| n||_4||\na^2c||_4\\
\leq&-\frac{1}{2}{||\na ^2 n||_2^2}+C||\na^2 n||_2||\na n||_2||\na u||_2+ ||\na^2 n||_2||\na n||_2||\na c||_\infty+C||\na^2n||_2|| n||_2||\na n||_2 \\
\leq&-\frac{1}{2}{||\na ^2 n||_2^2}+\frac{1}{4}||\na^2 n||_2^2+C\left(||\na u||_2^2+ ||\na c||_\infty^2+||n||_2^2\right)||\na n||_2^2\\
\leq&-\frac{||\na n||_2^4}{4C_{GNS}||n||_2^2}+C\left(C_ {u;H^1}^2+C_{\na c;\infty}^2+C_{L^2}^2 \right)||\na n||_2^2.
\end{align*} 
Now by standard ODE theory , we obtain that
\begin{align*} 
 ||\na n(t)||_2^2\leq& C\left(C_ {u;H^1}^2+C_{\na c;\infty}^2+C_{L^2}^2 \right)C_{L^2}^2+||\na n_0||_2^2,\quad \forall t\in[0,T].
\end{align*}
Combining this with \eqref{Proof_of_n_L2},  \eqref{na_c_L_infty} and  \eqref{na_u_L2} yields 
\begin{align}\label{H_1_bound}
||\na n(t)||_2+||\na u(t)||_2\leq&C_{H^1}(
C_{L\log L},C_{u;L^2},||n_0||_{L^1\cap L^\infty},
 ||n_0||_{H^1},||u_0||_{H^1})<\infty,\quad \forall t\in[0,T].
\end{align}

An iteration argument yields the $H^s\,\, (s\geq 2, s\in \mathbb{N})$ estimates. To set up the iteration, we make the following assumption 
\begin{align}\label{Assumption}
|| n(t)||_{H^{s-1}}+|| u(t)||_{H^{s-1}}\leq&C_{H^{s-1}}(
C_{L\log L},C_{u;L^2},||n_0||_{L^1\cap L
^\infty}, ||n_0||_{H^{s-1}},||u_0||_{H^{s-1}})<\infty,\quad \forall t\in[0,T],
\end{align}
and prove that \begin{align}\label{Iteration_conclusion}
|| n(t)||_{H^{s}}+|| u(t)||_{H^{s}}\leq&C_{H^{s}}(
C_{L\log L},C_{u;L^2},||n_0||_{L^1\cap L
^\infty}, ||n_0||_{H^{s}},||u_0||_{H^{s}})<\infty,\quad \forall t\in[0,T].
\end{align} Since we have already obtained the $H^1$ bound of the solution $(n,u)$, by iterating this argument, one can propagate any $H^s$-Sobolev norm as long as the conditions \eqref{S+n_condition} and \eqref{Energy_u_condition} are satisfied. 

We focus on the estimate of the density $n$ first. Applying the density equation \eqref{pePKS-NS}, the time evolution of the $\dot H^s$ semi-norm of $n$ can be expressed using integration by parts as follows
\begin{equation}
    \frac{1}{2}\frac{d}{dt} \sum_{|\al|=s}||\pa_x^\al n||^2_{2}+\sum_{|\al|=s}||\na \pa_x^\al n||^2_{2}=-\sum_{|\al|=s}\int  \pa_x^\al n \pa_x^\al( u \cdot \nabla n)dx-\sum_{|\al|=s}\int \pa_x^\al n \pa_x^\al \nabla \cdot (\nabla c n)dx=: \mathcal{I}_n+\mathcal{II}_n.\label{In_IIn}
\end{equation}
Now we estimate the first term $\mathcal{I}_n$ in \eqref{In_IIn}. We further decompose it into two parts: 
\begin{align}\label{In123}
\mathcal{I}_n=\sum_{|\al|=s}\int (\pa_x^\al n  ) u\cdot\nabla(\pa_x^\al n) dx+ \sum_{|\al|=s}\sum_{(0,0)<\beta\leq\al}\left(\begin{array}{cc}\beta_1\\\al_1\end{array}\right) \left(\begin{array}{cc}\beta_2\\\al_2\end{array}\right) \int \pa_x^\al n (\pa_x^\beta u) \cdot \na ( \pa_x^{\al-\beta}n )dx
=: \mathcal{I}_{n;1}+\mathcal{I}_{n;2}
.
\end{align}
The divergence-free property of the vector field $u$ and  integration by parts yield the vanishing of the first term $\mathcal{I}_{n;1}$ in \eqref{In123}, i.e., 
\begin{align}    
\mathcal{I}_{n;1}=\sum_{|\al|=s}\int u\cdot \nabla \left(\frac{|\pa_x^\al n|^2}{2}\right)dx=-\sum_{|\al|=s}\int( \nabla\cdot u) \left(\frac{|\pa_x^\al n|^2}{2}\right)dx=0.\label{I_n_1}
\end{align}
To estimate the second term $\mathcal{I}_{n;{2}}$ in \eqref{In123}, we first apply the H\"older inequality to obtain that
\begin{align}
\mathcal{I}_{n;2} \leq &\sum_{\substack{(0,0)<\beta\leq\al,\\|\al|=s}}\left(\begin{array}{rr}\beta_1\\{\al_1}\end{array}\right) \left(\begin{array}{rr}\beta_2\\{\al_2}\end{array}\right) \int \pa_x^\al n   \na(\pa_x^{\al-\beta} n)\pa_x^{\beta} u dx\\
 \leq& \sum_{\substack{(0,0)<\beta\leq\al,\\|\al|=s}}\left(\begin{array}{rr}\beta_1\\{\al_1}\end{array}\right) \left(\begin{array}{rr}\beta_2\\{\al_2}\end{array}\right)  ||n||_{\Dot{H}^s} ||\na \pa_x^{\al-\beta} n||_{L^p} ||\pa_x^{\beta}u||_{L^q} ,\quad  \frac{1}{p}+\frac{1}{q}=\frac{1}{2}.
\end{align}
Applying the Gagliardo-Nirenberg-Sobolev inequalities yields the following bounds
\begin{align}
||\na\pa_x^{\al-\beta} n||_{L^p}\leq& C_{GNS} ||n||_{\Dot{H}^{s+1}}^{\theta_1}||n||_{\Dot{H}^1}^{1-\theta_1},\quad \theta_1=\frac{|\al|-|\beta|+1-\frac{2}{p}}{s};\\
||\pa_x^{\beta} u||_{L^q}\leq &C_{GNS}||u||_{\Dot{H}^{s+1}}^{\theta_2}||u||_{\Dot{H}^1}^{1-\theta_2},\quad\theta_2=\frac{|\beta|-\frac{2}{q}}{s}=1-\theta_1. 
\end{align}
Combining these two estimates, the $H^1$ estimate \eqref{H_1_bound} with the previous estimation, and applying the Young's inequality yield that
\begin{align}
\mathcal{I}_{n;2} \leq &C_{GNS}||n||_{\dot H^{s}}\left(||n||_{\dot H^{s+1}}+||u||_{\dot H^{s+1}}\right)C_{H^1}.
\end{align} 
Combining this inequality and the $\mathcal{I}_{n;1}$ estimate \eqref{I_n_1} and the decomposition \eqref{In123} yields the estimate
\begin{align}
\mathcal{I}_n\leq &C||n||_{\dot H^{s}}\left(||n||_{\dot H^{s+1}}+||u||_{\dot H^{s+1}}\right)C_{H^1}\\
\leq&\frac{1}{8}||n||_{\dot H^{s+1}}^2+\frac{1}{8}||u||_{\dot H^{s+1}}^2+C(C_{H^{s-1}})||n||_{\dot H^s}^2.\label{I_n}
\end{align}
This completes the estimation of the $\mathcal{I}_n$ in \eqref{In_IIn}. 
Next we estimate the integral $\mathcal{II}_n$ in \eqref{In_IIn} as follows:  
\begin{equation}
\mathcal{II}_n=\sum_{|\al|=s}\int \na(\pa_x^\al n)\cdot \pa_x^{\al}( n \nabla c)dx\leq C ||n||_{\dot H^{s+1}}||n \na c||_{\dot H^s}.
\end{equation}
Now by the product estimate for Sobolev functions, the chemical gradient estimate \eqref{na_c_L_infty}, the $L^p$ bound on the cell density \eqref{C2infty}, the assumption \eqref{Assumption} and the $L^2$-boundedness of the Riesz transform, we have that 
\begin{align}
\mathcal{II}_n\leq& C||n||_{\dot H^{s+1}}(||n||_{H^s}||\na c||_{L^\infty}+||\na c||_{H^s}||n||_{L^\infty})\\
\leq&\frac{1}{8}||n||_{\dot H^{s+1}}^2+C(\CC)||n||_{\dot H^s}^2+C(C_{H^{s-1}}, \CC).\label{II_n}
\end{align}
Combining the $\mathcal{I}_n$ estimate \eqref{I_n}, the $\mathcal{II}_n$ estimate \eqref{II_n} and the equation \eqref{In_IIn}, we obtain that there exists a constant $C$ depending on the ${H^{s-1}}$ norm of the solution $(n,u)$ \eqref{Assumption} and the $L^p$ estimate of $n$ \eqref{C2infty} such that the following inequality holds:
\begin{equation}\label{I_n+II_n_Conclusion}
    \frac{1}{2}\frac{d}{dt} \sum_{|\al|=s}||\pa_x^\al n||^2_{2}+\frac{1}{2}\sum_{|\al|=s}||\na \pa_x^\al n||^2_{2}\leq \frac{1}{8}||u||_{\dot H^{s+1}}^2+
    C(C_{H^{s-1}},\CC)||n||_{\dot H^{s}}^2 +C(C_{H^{s-1}}, \CC).
\end{equation}

Next we focus on the $H^s$ estimate of $u$. Direct calculation with the velocity equation \eqref{eqn:flow} yields that
\begin{align}
    \frac{1}{2}\frac{d}{dt}\sum_{|\al|=s}||\pa_x^\al u||^2_{2}
+\sum_{|\al|=s}||\na \pa_x^\al u||^2_{2}=-\sum_{|\al|=s}\int \pa_x^\al u \cdot \pa_x^\al B(u, u)dx+ \int \pa_x^\al u \cdot\mathbb{P}\pa_x^\al(n\nabla c)dx=:\mathcal{I}_u+\mathcal{II}_u.\label{Iu_IIu}
\end{align}
Now we estimate each term in the decomposition \eqref{Iu_IIu}. For the $\mathcal{I}_u$ term, we decompose it into three terms as follows
\begin{align}
    \mathcal{I}_u=&\sum_{|\al|=s}\int (\pa_x^\al u)\cdot ((u\cdot \nabla ) \pa_x^\al u)dx+\sum_{|\al|=s}\sum_{\substack{\beta< \al\\ |\beta|\geq 1}}\left(\begin{array}{rr}\beta_1\\ \al_1\end{array}\right)\left(\begin{array}{rr}\beta_2\\ \al_2\end{array}\right)\int (\pa_x^\al u) \cdot((\pa_x^\beta u\cdot \na )\pa_x^{\al-\beta}u)dx\nonumber\\
    &+\sum_{|\al|=s}\int \pa_x^\al u \cdot((\pa_x^\al u \cdot \nabla )u)dx\nonumber\\
    =:&\mathcal{I}_{u;1}+\mathcal{I}_{u;2}+\mathcal{I}_{u;3}.\label{Iu123}
\end{align}
Now we estimate each term in the decomposition \eqref{Iu123}.
For the first term in \eqref{Iu123}, we apply the divergence-free property of the vector field  $u$ to obtain
\begin{equation}
    \mathcal{I}_{u;1}=\sum_{|\al|=s}\int u\cdot \nabla\left(\frac{|\pa_x^\al u|^2}{2}\right)dx=0.
\end{equation}
For the second term in \eqref{Iu123}, direct application of the H\"older inequality yields that
\begin{align*}
    \mathcal{I}_{u;2}\leq&\sum_{|\al|=s}\sum_{\substack{\beta< \al,\\|\beta|\geq 1}}\left(\begin{array}{rr}\beta_1\\\al_1\end{array}\right)\left(\begin{array}{rr}\beta_2\\\al_2\end{array}\right)\int |\pa_x^\al u| |\pa_x^\beta u| |\na \pa_x^{\al-\beta}u|dx \\
    \leq& C\sum_{|\al|=s}\sum_{\substack{\beta< \al,\\|\beta|\geq 1}}||u||_{\Dot{H}^{s}}||\pa_x^\beta u||_{{p}}||\na \pa_x^{\al-\beta}u||_{{q}},\quad \frac{1}{{p}}+\frac{1}{{q}}=\frac{1}{2}.
\end{align*}
Now we recall the following Gagliardo-Nirenberg-Sobolev inequalities 
\begin{align}
||\pa_x^\beta u||_{L^{p}}\leq& C_{GNS} ||u||_{\Dot{H}^{s+1}}^{\theta_3}||u||_{\Dot{H}^1}^{1-\theta_3},\quad \theta_3=\frac{|\beta|-\frac{2}{{{p}}}}{s};\\
||\na \pa_x^{\al-\beta} u||_{L^{q}}\leq &C_{GNS}||u||_{\Dot{H}^{s+1}}^{\theta_4}||u||_{\Dot{H}^1}^{1-\theta_4},\quad \theta_4=\frac{(|\al|-|\beta|+1)-\frac{2}{{{q}}}}{s}=1-\theta_3.
\end{align} 
Combining these inequalities and the estimation above yields that 
\begin{align}
 \mathcal{I}_{u;2}\leq C_{GNS}||u||_{\Dot{H}^{s+1}}||u||_{\Dot{H}^{s}}||u||_{\Dot{H}^{1}}.
\end{align}
Now we estimate the last term $\mathcal{I}_{u;3}$ in the decomposition \eqref{Iu123} using the H\"older inequality and the Gagliardo-Nirenberg-Sobolev inequality as follows\begin{equation}
    \mathcal{I}_{u;3}\leq C\sum_{|\al|=s}||u||_{\Dot{H}^s}||\pa_x^\al u||_{L^4} ||\na u||_{L^4} \leq C_{GNS}||u||_{\Dot{H}^1}||u||_{\Dot{H}^{s+1}}||u||_{\Dot{H}^{s}}.
\end{equation}
Combining the estimations of the $\mathcal{I}_{u;1},$ $\mathcal{I}_{u;2}$ and $\mathcal{I}_{u;3}$ terms above and the decomposition \eqref{Iu123}, and applying the Young's inequality yield the following
\begin{align}
    \mathcal{I}_u
    \leq&\frac{1}{8}||u||_{\dot H^{s+1}}^2+C(C_{H^{s-1}},C_{1,\infty})||u||_{\dot H^{s}}^2+C(C_{H^{s-1}},\CC).\label{I_u}
\end{align} 
Now we estimate the term $\mathcal{II}_u$ in \eqref{Iu_IIu} with the product estimate for Sobolev functions, the chemical gradient estimate \eqref{na_c_L_infty}, the $L^p$ bound on the cell density $n$ \eqref{C2infty}, the iteration  assumption  \eqref{Assumption}, the divergence free property of the vector field $u$, the fact that projection $\mathbb{P}$ is self-adjoint, and the $L^2$-boundedness of the Riesz transform as follows
\begin{align}
    &\mathcal{II}_u\leq || u||_{\dot H^{s}} || n\nabla c||_{\dot{H}^s}\leq C
    || u||_{\dot H^{s}} (||n||_{H^s}||\na c||_\infty+|| n||_\infty||\nabla c||_{{H}^s})\\ &\leq||  u||_{\dot H^s}^2+ C ||n||_{ H^s}^2||\na c||_{\infty}^2+C||n||_\infty^2||\na c||_{H^s}^2 
    \leq ||  u||_{\dot H^s}^2+ C(\CC) ||n||_{\dot H^s}^2+C(C_{H^{s-1}},\CC).\label{II_u}
\end{align}
Combining the estimates for $\mathcal{I}_u$ \eqref{I_u} and $\mathcal{II}_u$ \eqref{II_u}, and the decomposition 
 \eqref{Iu_IIu}, we end up with the estimate on the time evolution of the $\dot H^s$ seminorm of vector field $u$ 
\begin{align}\label{Iu+IIu_conclusion}
 \frac{1}{2}\frac{d}{dt}\sum_{|\al|=s}||\pa_x^\al u||^2_{2}
+\frac{1}{2}\sum_{|\al|=s}||\na \pa_x^\al u||^2_{2}\leq C(C_{H^{s-1}},\CC)(||u||_{\dot H^s}^2+||n||_{\dot H^s}^2)+C(C_{H^{s-1}},\CC).
\end{align}
Finally, combining the estimates \eqref{I_n+II_n_Conclusion} and \eqref{Iu+IIu_conclusion}, we have that 
\begin{align}
    \frac{1}{2}&\frac{d}{dt}\left(||u||^2_{\Dot{H}^s}+||n||^2_{\Dot{H}^s}\right) \\
    \leq&-\frac{1}{4}||n||^2_{\Dot{H}^{s+1}}-\frac{1}{4}||u||^2_{\Dot{H}^{s+1}}+C(C_{H^{s-1}},\CC)(||n||_{\Dot {H}^{s}}^2+||u||_{\Dot {H}^{s}}^2)+ C(C_{H^{s-1}},\CC).\label{Hs_conclusion_0}
\end{align}
Applying the Gagliardo-Nirenberg-Sobolev inequality, we end up with the following
\begin{align}
    -||f||_{\Dot{H}^{s+1}}^2\leq -\frac{||f||_{\Dot{H}^s}^{2+\frac{2}{s}}}{C_{GNS}||f||_{L^2}^{\frac{2}{s}}}.
\end{align}
Applying this upper bound on the dissipative terms appeared in \eqref{Hs_conclusion_0} and recalling the $L^p$ estimate \eqref{C2infty} and the $L^2$ energy condition of the vector fields $u$ \eqref{Energy_u_condition}, we obtain that
\begin{align}
    \frac{1}{2}\frac{d}{dt}(||u||^2_{\Dot{H}^s}+||n||^2_{\Dot{H}^s}) \leq&-\frac{||n||_{\Dot{H}^{s}}^{2+\frac{2}{s}}}{4C_{GNS} \CC^{\frac{2}{s}}}-\frac{||u||_{\Dot{H}^{s}}^{2+\frac{2}{s}}}{4C_{GNS}C_{u;L^2}^{\frac{2}{s}}}+C(C_{H^{s-1}},\CC)(||n||_{\Dot {H}^{s}}^2+||u||_{\Dot {H}^{s}}^2)+ C(C_{H^{s-1}},\CC).
\end{align}
Therefore we have that
\begin{align*}
||n(t)||_{H^{s}}+||u(t)||_{H^s}\leq C_{H^s}(||n_0||_{H^s},||u_0||_{H^s}, C_{H^{s-1}},\CC,C_{u;L^2})<\infty, \quad \forall t \in [0,T]. 
\end{align*}
This concludes the proof. 
\end{proof}

Next we prove Theorem \ref{thm:Second_moment_criterion}.
\begin{proof}[Proof of Theorem  \ref{thm:Second_moment_criterion}]
The proof involves two steps. First we estimate the entropy
\begin{align}
S[n]=\int_{\rr^2}n\log n dx.
\end{align} 
Then we estimate its negative part $S^-[n]$ through second moment bound. Since $S^+[n]=S[n]+S^-[n]$, these estimates yield the bound on the positive part of the entropy $S^+[n]$.

To estimate the entropy, we combine the decay estimate of the free energy \eqref{Free_energy_PKSNS} and the following logarithmic Hardy-Littlewood-Sobolev inequality (see e.g. \cite{CarlenLoss92}):
\begin{thm}[Logarithmic Hardy-Littlewood-Sobolev Inequality] For all nonnegative functions $f \in L^1(\rr^2)$ such that $f\log f$ and $f\log(1+|x|^2)$ belong to $L^1(\rr^2)$, there exists a constant $C(M)$ such that the following inequality holds
\begin{align}\label{log HLS}
\int_{\rr^2}f\log f dx+\frac{2}{M}\iint_{\rr^2\times\rr^2} f(x)f(y)\log|x-y|dxdy\geq -C(M),\quad M =\int_{\rr^2}fdx>0.
\end{align}
\end{thm}
Combining \eqref{log HLS} and Lemma \ref{Lem:Free_Energy_Decay} yields that
\begin{align*}
E[n_0,u_0]\geq& E[n,u]\\
=&\left(1-\frac{M}{8\pi}\right)\int_{\rr^2} n\log n dx+\frac{M}{8\pi}\left(\int_{\rr^2}n\log n dx+\frac{2}{M}\iint_{\rr^2\times \rr^2} n(x)\log|x-y| n(y)dxdy\right)+\frac{||u||_2^2}{2}\\
\geq &\left(1-\frac{M}{8\pi}\right)S[n]-\frac{M}{8\pi}C(M)+\frac{||u||_2^2}{2}.
\end{align*} 
As a result, we obtain an a-priori bound on the entropy $S[n]$ and the $L^2$ norm of the velocity $||u||_2$ for any finite time
\begin{align}\label{CT_definition}
\frac{||u(t)||_2^2}{2(1-\frac{M}{8\pi})}+S[n(t)]\leq \frac{E[n_0,u_0]+\frac{M}{8\pi} C(M)}{1-\frac{M}{8\pi}}\leq C(M, E[n_0,u_0])<\infty,\quad \forall t\in[0,T].
\end{align}
Therefore, we obtain the bound on the entropy $S[n]$ and the energy $||u||_2^2$. 

Next we estimate the negative part of the entropy $S^-[n]$. To this end, we recall the following inequality 
\begin{align}\label{S-n_control}
\int_{\rr^2} g\log^- gdx\leq \frac{1}{2}\int_{\rr^2} g |x|^2dx+\log(2\pi)\int_{\rr^2} g dx +\frac{1}{e}, \quad g\geq 0, 
\end{align} 
whose proof can be found in Lemma 2.2, \cite{BlanchetCarrilloMasmoudi08}. Since the second moment is assumed to be bounded \eqref{Second_Moment_Bound}, direct application of the inequality yields the following estimate:
\begin{align}\label{CLlogL_defn}
||u(t)||_2+\int_{\rr^2} n(t,x)\log ^+ n(t,x) dx\leq C(C_V, E[n_0, u_0], M)<\infty,
\end{align}
on the interval $[0,T]$. Now all the conditions in Theorem \ref{thm:S+n_condition} are checked, and this concludes the proof of Theorem \ref{thm:Second_moment_criterion}. 
\end{proof}

\begin{proof}[Proof of Corollary \ref{cor:radially_symmetric_case}] It is enough to show that if the initial data $(n_0(x), u_0(x))$ is radially symmetric,  then the second moment is bounded for any finite time, i.e.,
\begin{align}\label{proof_cor}
\int n(t,x)|x|^2 dx\leq \int n_0(x)|x|^2dx+4M t.
\end{align} Explicit calculation of the time evolution of the second moment yields that
\begin{align}\label{proof_cor_1}
\frac{d}{dt}\int n(t,x)|x|^2dx=4M-\frac{1}{2\pi}M^2-\int x^2 \na \cdot(un)dx.
\end{align}
To estimate the last term in the above equality, we will rewrite it in a different form. To this end, we introduce the stream function of the velocity field $u$, 
\[\phi:= 
{ \de^{-1}\mathrm{curl} u},\quad (-\pa_{x_2},\pa_{x_1})\phi=u.
\] Since the equation \eqref{pePKS-NS} preserves radial symmetry, the solutions $(n,u)$ are radially symmetric. As a result, the stream functions $\phi$ are also radially symmetric, which implies $(x_1\pa_{x_2}-x_2\pa_{x_1})\phi\equiv0$. Applying these facts, we rewrite the last term in the time evolution of the second moment in the following manner,
\begin{align}
\int |x|^2 \na \cdot(un)dx=-2\int x\cdot u n dx=-2\int x\cdot \na^{\perp}\phi ndx=2\int (x_1\pa_{x_2}-x_2\pa_{x_1})\phi ndx=0.
\end{align}
Combining this and \eqref{proof_cor_1} yields \eqref{proof_cor}. Since the second moment condition \eqref{Second_Moment_Bound} is checked, Theorem \ref{thm:Second_moment_criterion} can be applied. This completes the proof of the first part of Corollary \ref{cor:radially_symmetric_case}.

{If the total mass is greater than $8\pi$, then by the same argument as above, we observe that
\begin{align}
\frac{d}{dt}\int n(t,x)|x|^2dx=4M-\frac{1}{2\pi}M^2<0.
\end{align}
Hence if the solution $(n,u)$ is regular on the time interval $\dss [0,T_\star],\,T_\star:=\frac{8\pi}{4M(M-8\pi)}\int n_0|x|^2dx$, then the second moment becomes zero at time $T_\star$, which is impossible. Hence the solution must blow up on or before time $T_\star$. This concludes the proof of the second part of Corollary \ref{cor:radially_symmetric_case}. } 
\end{proof}

Now we introduce the modified free energy $\cF_{\Gamma}$ and its properties. We introduce  the following modified free energy:
\begin{align}
\cF_{\Gamma}[n,u]=\int n\Gamma(n)-\frac{nc}{2}+\frac
{|u|^2}{2}dx,\label{def:En0}
\end{align}
where $\Gamma $ is defined as
\begin{align}
\Gamma(n)=\left\{\begin{array}{rr}\log n , \quad n\geq \eta;\\
\log \eta+\eta^{-1}\left(n-\eta\right)-\frac{\eta^{-2}}{2}\left(n-\eta\right)^2,\quad n<\eta.\end{array}\right.\quad
\eta:=\eta(\delta,M)= \min\left\{1,\frac{\delta}{M}\right\}.\label{Gamma}
\end{align}
The $\Gamma$ function is chosen such that it matches $\log$ when $n$ is large but is bounded from below when $n$ is small. Here, we have replaced the function $\log (\eta+(n-\eta))$ by its degree two Taylor expansion centred at $\eta$ when $n<\eta$ and use the original $\log$ function when $n\geq \eta $.

The next lemma states that the modified free energy \eqref{def:En0} grows at most linearly under the dynamics \eqref{pePKS-NS}. 
\begin{lem}\label{Lemma:free energy evolution} The time derivative of the modified free energy $\cF_{\Gamma}[n,u]$, defined in \eqref{def:En0}, satisfies the following estimate:
\begin{equation}\label{free energy evolution}
\frac{d}{dt}\cF_{\Gamma}[n(t),u(t)] \leq \delta ,\quad\forall t\in [0,\infty).
\end{equation}
Furthermore, the following quantity is bounded:
\begin{align}\label{Gamma-control}
-\int_{n<1}n\Gamma(n) dx \leq \left(-\log\eta(\delta,M)+\frac{3}{2}\right)M.
\end{align}
\end{lem}
\begin{proof}
Taking the time derivative of $\cF_{\Gamma}[n(t),u(t)]$,  applying the divergence-free condition of the vector field $u$ and integration by parts yield
\begin{align}
\frac{d}{dt}&\left(\int n\Gamma( n)-\frac{nc}{2}+\frac
{|u|^2}{2}dx\right)\nonumber\\
=&\int (n)_t(\Gamma (n)- c)dx+\int n(\Gamma (n))_tdx+\int u\cdot u_tdx\nonumber\\
=&-\int (n\na \log n-\na cn)\cdot(\ga '(n)\na n-\na c)dx-\int u\cdot \na n \ga(n )dx+\int \na \cdot( un) cdx\nonumber\\
&-\int\na(n\ga'(n))\cdot(n\na \log n -\na c n)dx-\int u\cdot \na n n\ga'(n)dx-\int |\na u|^2 dx+\int nu\cdot\na c dx
\nonumber\\
=:&\sum_{i=1}^7T_i\label{Gamma-control-1}.
\end{align}
Applying the integration by parts, we have that the third term $T_3$ and the seventh term $T_7$ in\eqref{Gamma-control-1} cancel each other. Now we consider the second term $T_2$ and the fifth term $T_5$. Since the $\Gamma$ function is finite near the origin, we define the following functions:
 \begin{align*}
 \mathcal{E}(r)=\int_0^r \ga(s)ds,\quad
\mathcal{G}(r) =\int_0^r s\ga'(s)ds.
 \end{align*}
 The second term $T_2$ and fifth term $T_5$ can be explicitly calculated using the divergence free condition $\na \cdot u=0$ and integration by parts as follows:
\begin{align*}
T_2=&-\int u\cdot \na(\mathcal{E})dx=\int (\na\cdot u)\mathcal{E}dx=0;\\
T_5=&-\int u\cdot\na (\mathcal{G})dx=\int(\na\cdot u)\mathcal{G}dx=0.
\end{align*}
Next we estimate the terms $T_1+T_4$
.
Applying the definition of $\Gamma$ \eqref{Gamma},
{the cut-off threshold $\eta=\min\{\frac{\delta}{M^2},1\}$, and the fact that $\Gamma'(n)=2\eta^{-1}-\eta^{-2}n$ for $n\leq \eta$, direct calculation yields the following equality 
\begin{align*}
T_1+T_4
=&-\int_{n\geq \eta}(n\na \log n-\na c n)\cdot\left(\frac{1}{n}\na n-\na c\right)dx \\ 
&  -\int_{n<\eta}(n\na \log n-\na c n)\cdot\left(\left(2\eta^{-1}-\eta^{-2}n\right)\na n-\na c\right)dx\\
& -\int_{n<\eta}\left(2\eta^{-1}-\eta^{-2}n\right)\na n\cdot(n\na \log n-\na c n)dx+ \int_{n<\eta}n\eta^{-2}\na n \cdot(n\na \log n-\na c n)dx.
\end{align*}}
Notice the following inequality:
\begin{align}\sup_{n<\eta}\sqrt{\left(-3\eta^{-2}n+4\eta^{-1}\right)n}\leq \frac{2}{\sqrt{3}}<2,
\end{align}
which implies,
\begin{align}
T_1&+T_4\\=&-\int_{n\geq \eta}n|\na \log n-\na c|^2dx-\int_{n<\eta}\left(4\eta^{-1}-3\eta^{-2}n\right)|\na n|^2dx\\
&+\int_{n<\eta}\sqrt{\left(-3\eta^{-2}n+4\eta^{-1}\right)n}\sqrt{\left(-3\eta^{-2}n+4\eta^{-1}\right)n}\na c \cdot \na ndx-\int_{n<\eta}n|\na c|^2 dx+\int_{n<\eta} \na n\cdot\na c dx \\
\leq &-\int_{n\geq \eta}n|\na\log n-\na c|^2dx -\int_{n<\eta}\left(4\eta^{-1}
-3\eta^{-2}n\right)|\na n|^2dx\\
&+\frac{2}{\sqrt{3}}\int_{n<\eta}\sqrt{\left(-3\eta^{-2}n+4\eta^{-1}\right)n}|\na c||\na n|dx -\int_{n<\eta}n|\na c|^2 dx+\int_{n<\eta} \na n\cdot\na c dx.%
\end{align}
Completing a square using the 2nd, 3rd, 4th terms in the last line yields
\begin{align}
T_1+T_4
\leq &-\int_{n\geq \eta}n|\na \log n-\na c|^2dx-\frac{2}{3}\int_{n<\eta}\left(4{\eta^{-1}}-3\eta^{-2}n\right)|\na n|^2dx\\
&-\int_{n <\eta} \left(\sqrt{4\eta^{-1}-3\eta^{-2}n}\frac{1}{\sqrt{3}}|\na n|-\sqrt{n}|\na c|\right)^2dx +\int_{n<\eta} \na n\cdot\na c dx.\label{3D zero mode T 0 1}
\end{align}
\noindent
\textbf{Claim:} The following estimate holds
\begin{align}
\int_{n<\eta}\na n\cdot \na c dx\leq \delta.\label{Claim}
\end{align}
To prove the claim, we make the qualitative assumption that $n\in C^\infty(\rr^2)\cap H^s(\rr^2),\quad s\geq 3$. However, the final estimate will be independent of the higher regularity norms of the densities $n$ and $c$. We apply the choice of $\eta$ \eqref{Gamma} and integration by parts to obtain
\begin{align}
\int_{n<\eta}\na n\cdot\na c dx=\int \na(\min\{n,\eta\})\cdot \na c dx=-\int \min\{n,\eta\} \de c dx\leq \int \eta n dx\leq \eta M\leq\delta.
\end{align}  Here we have applied the equality $\na n \mathbf{1}_{n< \eta}=\na (\min\{n,\eta\})$ almost everywhere if $n\in W^{1,p}(\rr^2),$ for $1<p<\infty$. This is a natural consequence of Exercise 17 in Evans \cite{Evans} Chapter 5. To explicitly justify the integration by parts, one can use positive $C_c^\infty$ function to approximate the $W^{1,4/3}$ function $\min\{n,\eta\}$ and the $W^{1,4}$ function $\na c$.   

Therefore, combining the claim and estimate \eqref{3D zero mode T 0 1}, we deduce that,
\begin{align}
T_1+T_4\leq& -\int_{n\geq \eta}n|\na\log n-\na c|^2dx-\frac{2}{3}\int_{n<\eta}\left(4\eta^{-1}-3\eta^{-2}n\right)|\na n|^2dx +\delta\leq\delta.\label{3D zero mode T 0}
\end{align}
This finishes the treatment of all $T_i$'s in (\ref{Gamma-control-1}). Therefore, the estimate \eqref{free energy evolution} follows.  

Estimate \eqref{Gamma-control} follows from the fact that the function $\Gamma$ is bounded from below by $\log \eta(\delta,M)-\frac{3}{2}$, which is a finite number. This finishes the proof of Lemma \ref{Lemma:free energy evolution}.
\end{proof}
\begin{proof}[Proof of Theorem \ref{thm:modified_free_energy_result}]
We rewrite the approximate free energy so that the inequality (\ref{log HLS}) can be applied:
\be\ba
 E_{\Gamma}[n_{0},u_0]+\delta t
\geq &\int n\Gamma(n)dx-\int\frac{nc}{2}dx+\int\frac{1}{2}|u|^2dx\\
=&\int n\log^+ndx+\int_{n< 1}n\Gamma(n)dx+\frac{1}{4\pi}\iint \log|x-y|n(x)n(y)dxdy+\frac{1}{2}||u||_2^2\\
=&\left(1-\frac{M}{8\pi}\right)\int n\log^+ndx+\int_{n< 1}n\Gamma(n)dx\\
&+ \frac{M}{8\pi}\left(\int n\log^+ndx+\frac{2}{M}\iint \log|x-y|n(x)n(y)dxdy \right)+\frac{1}{2}||u||_2^2.
\ea\ee
Applying the log-HLS (\ref{log HLS}) and \eqref{Gamma-control} yields:
\be\ba
 E_{\Gamma}[n_{0}, u_0]+\delta t\geq& \left(1-\frac{M}{8\pi}\right)\int n\log^+ndx+\int_{n<1}n\Gamma(n)dx- C(M)\frac{M}{8\pi}+\frac{1}{2}||u||_2^2\\
\geq& \left(1-\frac{M}{8\pi}\right)\int n\log^+ndx-M\log\eta(\delta,M)^{-1}-\frac{3}{2}M-C(M) \frac{M}{8\pi}+\frac{1}{2}||u||_2^2,
\ea\ee
which leads to a bound on the positive part of the entropy $S^+[n(t)]$ and the fluid energy  $||u||_2^2$ for any finite time, i.e.
\begin{align}
\left(1-\frac{M}{8\pi}\right)\int n\log^+ndx+\frac{1}{2}||u||_2^2\leq \cF_\Gamma[n_{0},u_0]+\delta t+M\log\eta({\delta},M)^{-1}+\frac{3}{2}M+C(M)\frac{M}{8\pi}.
\end{align}
This yields that
\begin{align*}
S^+[n(t)]+||u(t)||_2^2< C(E_\Gamma[n_0,u_0],M,\delta)+\delta t<\infty,\quad \forall t\in [0,\infty).
\end{align*}
This concludes the proof of Theorem \ref{thm:modified_free_energy_result}. 
\end{proof}
\begin{proof}[Proof of Theorem \ref{thm:R2}]
Now we highlight the adjustment in the remaining part of the proof of Theorem \ref{thm:R2} comparing to the proof of Theorem \ref{thm:S+n_condition}. The main adjustment takes place in the proof of the $L^2$ norm of the cell density  \eqref{Proof_of_n_L2}. Since the positive part of the entropy $S^+[n(t)]$ is growing linearly with rate $\delta$, the quantity $\eta_K=||(n-K)_+||_1$ will not be uniformly bounded on arbitrarily long interval as in \eqref{small_eta_K}. To overcome this difficulty, we adjust the vertical cut-off level $K$ as time progresses. Specifically speaking, we fix an arbitrary time interval $[0,T]$ and do estimation on it. First note that on this time interval, we have that 
\begin{align}
S^+[n(t)]+||u(t)||_2^2\leq C(E_{\Gamma}[n_0,u_0], M,\delta)+\delta T<\infty, \quad \forall t\in [0,T].
\end{align}
Now we choose the vertical cut-off level $K(T)$ such that the quantity $\eta_{K(T)}:=||(n-K(T))_+||_1$ is small in the sense that
\begin{align}
\eta_{K(T)}\leq \frac{C(E_\Gamma[n_0,u_0], M,\delta)+\delta T}{\log K(T)}\leq \frac{1}{8}C_{GNS},
\end{align}  
where $C_{GNS}$ is the universal constant appeared in the $L^2$ energy estimate \eqref{Proof_of_n_L_2_0}. The resulting $K(T)$ is larger than 
\begin{align}\label{KT}
K(T)\geq \exp\left\{\frac{8C(E_\Gamma[n_0,u_0],M, \delta)}{C_{GNS}}\right\}\exp\left\{\frac{8\delta T}{C_{GNS}}\right\}.
\end{align} Now combining the size of $K(T)$ and a direct $L^2$ energy estimation on the quantity $(n-K(T))_+$, which is the same as \eqref{Proof_of_n_L_2_0}, yields that
\begin{align*}
||n(t)||_{2}\leq &2||\min\{n(t), K(T)\}||_2+2||(n(t)-K(T))_+||_2\\
\leq& 2K(T)^{1/2}M^{1/2}+C(||n_0||_2,M) K(T)^{1/2}\leq C(||n_0||_2, E[n_0,u_0] ,M,\delta)e^{\frac{4\delta}{C_{GNS}} T},\quad \forall t\in[0,T].
\end{align*}
Since the time $T$ is arbitrary, we have that the $L^2$ norm of $n$ can grow at most exponentially with rate $\frac{4\delta}{C_{GNS}}$. Since $\delta$ is arbitrarily small, we abuse the notation and still denote the rate as $\delta$. The remaining part of the proof is similar to the proof of Theorem \ref{thm:S+n_condition},  so we omit the details. This concludes the proof of Theorem \ref{thm:R2}.
\end{proof}

\begin{rmk}With Theorem \ref{thm:R2} proven, we make a comment on the second moment $V[n(t)]$ \eqref{Second_moment}. We estimate the time evolution of the second moment as follows
\begin{align}
\frac{d}{dt}\int n|x|^2 dx=&4M-\frac{1}{2\pi}M^2-\int x^2 \na \cdot(un)dx=4M-\frac{1}{2\pi}M^2+2\int x \cdot undx\\
\leq & 4M+||u||_\infty M^{1/2}\left(\int n|x|^2dx\right)^{1/2}.
\end{align}
Since the $||u||_\infty$ is bounded on arbitrary finite time interval, the second moment is bounded for any finite time.  
\end{rmk}

{In the last part of this section, we consider the long time behavior of the \emph{radially symmetric solutions} to the equation \eqref{pePKS-NS} and prove Theorem \ref{thm:radial_long_time}. 

To prove Theorem \ref{thm:radial_long_time}, we first rewrite the equation of the velocity in the vorticity form and present some necessary lemmas. Recall that the vorticity \[\omega=\na^\perp\cdot  u=\pa_{x_1}u^2-\pa_{x_2}u^1,\quad  \na^{\perp}=(-\pa_{x_2},\pa_{x_1}) \] and the velocity $u$ is related through the Biot-Savart law:
\begin{align}
u(t,x)= \na^\perp \de^{-1}\omega(t,x)=\frac{1}{2\pi}\na^\perp\int_{\rr^2}\log|x-y|\omega(t,y)dy=:\na^\perp \psi(t,x),  
\end{align}
where  $\psi$ is the stream function. In the vorticity formulation, the equation \eqref{pePKS-NS} has the following form
\begin{align}\label{pePKS-NS-vorticity-form}
\left\{\begin{array}{rrrrr}\ba
\pa_t n+&u\cdot \na n=\de n-\na \cdot( n\na c),\quad -\de c=n,\\
\pa_t \omega+&u\cdot \na \omega=\de \omega +\na ^\perp\cdot(n\na c),\quad u=\na^\perp\de^{-1}\omega,\\
n(t=&0,x)=n_{0}(x),\quad \omega(t=0,x)=\omega_{0}(x), \quad x\in  \rr^2.\ea\end{array}\right.
\end{align} 
 To show long time decay of the solution, it is classical to consider the solutions in the self-similar variables:
\begin{align}
n(t,x)=&\frac{1}{R^2(t)}N\left(\log R(t),\frac{x}{R(t)}\right),\quad c(t,x)=C\left(\log R(t),\frac{x}{R(t)}\right),\quad R(t)=(1+2t)^{1/2};\\
\omega(t,x)=&\frac{1}{R^2(t)}\Omega\left(\log R(t),\frac{x}{R(t)}\right),\quad \psi(t,x)=\Psi\left(\log R(t), \frac{x}{R(t)}\right).
\end{align}
We further consider the new coordinate $ \tau:=\log R(t),\, X:=\frac{x}{R(t)}$, and rewrite the equation \eqref{pePKS-NS-vorticity-form} in the following form:
\begin{align}\label{pePKS-NS-self-similar}
\left\{\begin{array}{ccc}\ba\pa_\tau N-\na\cdot(XN)=&\de N-\na \cdot(N\na C)-\na^\perp\Psi\cdot \na N,\quad -\de C=N,\\
\pa_\tau\Omega-\na \cdot(X\Omega)=&\de \Omega+\na ^\perp\cdot(N\na C)-\na^\perp \Psi\cdot \na \Omega,\quad \Omega=\de\Psi,\\
N(0,X)=&n_0(x),\quad \Omega(0,X)=\omega_0(x).
\ea\end{array}\right.
\end{align}
If the solution is radially symmetric, then the second moment of $N$ is uniformly bounded in time. This is the content of the next lemma.
\begin{lem}
Consider radially symmetric solutions $(N,\Omega)\in \mathrm{Lip} _\tau([0,\frac{1}{2}\log(1+2T)];H^s(\rr^2)),\, s\geq 3$ 
to the equation \eqref{E_self-similar} subject to initial  constraints in Corollary \ref{cor:radially_symmetric_case}. The second moment of the solution is bounded in time
\begin{align}
\sup_{\tau\in[0,\infty)}\int_{\rr^2} N(\tau,X) |X|^2dX\leq C_{s;V}<\infty.\label{Uniform-in-time-second-moment-bound}
\end{align}
\end{lem}
\begin{proof}
The calculation is similar to the calculation in Corollary \ref{cor:radially_symmetric_case}. Direct calculation yields that
\begin{align*}
\frac{d}{d\tau}\int N(\tau, X)|X|^2dX=&4M-\frac{1}{2\pi}M^2 -2\int N(\tau, X)|X|^2dX-\int \na \cdot(\na^\perp\Psi (\tau,X) N(\tau,X))|X|^2dX.
\end{align*} 
Since the solutions are radially symmetric, the last term is zero. Now we see that the second moment is uniformly  bounded in time. 
\end{proof}
For the equation \eqref{pePKS-NS-self-similar}, if the solution $(N,\Omega)$ is radially symmetric, there is a dissipative free energy:
\begin{align}\label{E_self-similar}
E_{S}[N,\Omega]=\int_{\rr^2} N\log N-\frac{1}{2}NC+\frac{1}{2}N|X|^2-\frac{1}{2}\Psi\Omega dX.
\end{align}
This is the content of the following lemma:
\begin{lem}
Consider radially symmetric solutions $(N,\Omega)\in \mathrm{Lip} _\tau([0,\frac{1}{2}\log (1+2T)];H^s(\rr^2)), \, s\geq 3$ 
to the equation \eqref{E_self-similar} subject to initial finite second moment constraint in Corollary \ref{cor:radially_symmetric_case}. The free  energy \eqref{E_self-similar} is dissipative in the sense that $\frac{d}{d\tau }E_S[N(\tau), \Omega(\tau)]\leq 0$.
\end{lem}
\begin{proof}
Direct calculation yields that
\begin{align}
\frac{d}{d\tau}\int &N\log N-\frac{1}{2}NC+\frac{1}{2}N|X|^2-\frac{1}{2}\Omega\Psi dX\\
=&\int N_\tau \left(\log N- C+\frac{1}{2}|X|^2\right)dX-\int \Omega_\tau\Psi dX\\
=&\int (\na\cdot(N \na\log N)-\na \cdot(N\na C)+\na\cdot (NX)- \na\cdot( \na^\perp \Psi N))\left(\log N-C+\frac{1}{2}|X|^2\right)dX \\
&-\int (\de \Omega+\na\cdot(X\Omega) -\na \cdot(\na^\perp \Psi \Omega) +\na^\perp\cdot(N\na C))\Psi dX\\
=&-\int N|\na \log N-\na C+X|^2dX-\int \Omega^2 dX\\
&+\int N \na^\perp\Psi \cdot (\na \log N-\na C+X)dX-\int \na\cdot (X \Omega)\Psi dX+\int N\na C\cdot \na^\perp\Psi dX\\
=&:-\int N|\na \log N-\na C+X|^2dX-\int \Omega^2 dX+\sum_{\ell=1}^3 T_{s;\ell}.\label{Ts123}
\end{align}
By the fact that for the  radially symmetric solutions, $\na ^\perp \Psi $ is perpendicular to the vectors $X,\na C, \na N$, we have that $T_{s;1}=T_{s;3}=0$. For the $T_{s;2}$ term in \eqref{Ts123}, we have that 
\begin{align}
T_{s;2}=& \int X_1\pa_{X_1}\Psi (\pa_{X_1X_1}\Psi+\pa_{X_2X_2} \Psi)dX+\int X_2\pa_{X_2} \Psi (\pa_{X_1X_1}\Psi+\pa_{X_2X_2}\Psi) dX\\
=&\frac{1}{2}\int\pa_{X_1}(\pa_{X_1}\Psi)^2 X_1dX-\frac{1}{2}\int\pa_{X_1}(\pa_{X_2}\Psi)^2 X_1 dX\\
& -\frac{1}{2}\int\pa_{X_2}(\pa_{X_1}\Psi)^2X_2dX+\frac{1}{2}\int\pa_{X_2}(\pa_{X_2}\Psi)^2X_2dX\\
=&0.
\end{align} Combining these calculations and the relation \eqref{Ts123}, we have obtained that the free energy is decaying.
\end{proof}
\begin{proof}[Proof of Theorem \ref{thm:radial_long_time}]
Since the free energy is bounded and the second moment is bounded \eqref{Uniform-in-time-second-moment-bound}, through a standard argument involving logarithmic Hardy-Littlewood-Sobolev inequality, which is similar to the ones to prove Theorem  \ref{thm:Second_moment_criterion} and Theorem \ref{thm:S+n_condition}, we have that the solution $N,\Omega$ is uniformly bounded in time, i.e., $||N||_{L_\tau^\infty([0,\infty);L^2_X)}+||\Omega||_{L_\tau^\infty([0,\infty);L_X^2)}\leq C_s<\infty$. Now by the relation between the $L_{x}^2$ norm and the $L_{X}^2$  norm, we have that 
\begin{align}
||n(t)||_{L_{x}^2}^2+||\omega(t)||_{L_x^2}^2\leq \frac{1}{R^2(t)}\left(||N||_{L_\tau^\infty([0,\infty);L^2_X)}^2+||\Omega||_{L_\tau^\infty([0,\infty);L^2_X)}^2\right)=\frac{1}{1+2t}C_s^2,\quad \forall t\in[0,\infty).
\end{align} 
This concludes the proof of the theorem. 
\end{proof}
}
\section{Torus Case: $\mathbb{T}^2$}\label{Sec:T2}
Before we start the proof of Theorem \ref{thm:Torus}, we first collect some useful facts. Without loss of generality, we assume that the average of the velocity $u$ is zero, i.e.,
\begin{align*}
\frac{1}{|\mathbb{T}^2|}\int_{\mathbb{T}^2}u_0^i(x)dx=0,\quad i=1,2.
\end{align*}
The average-zero properties are propagated along the dynamics \eqref{pePKS-NS-Torus}. To check this, we calculate the time evolution of the mean using the divergence-free condition of $u$ and the elliptic equation of the chemical $c$ as follows: 
\begin{align*}
\frac{d}{dt}\int u^idx
=&-\sum_{j=1}^2\int u^j\pa_j u^i dx+\int n\pa_i c dx
=\int (\na\cdot u) u^idx+\int\left(- (\pa_1\pa_1+\pa_2\pa_2)c+\overline{n}\right)\pa_i cdx\\
=&\int \pa_1 c\pa_i \pa_1 c dx+\int \pa_2 c\pa_i\pa_2 cdx
=\sum_{j=1}^2\int \frac{1}{2}\pa_i(\pa_j c)^2dx=0.
\end{align*}
As a result, $\overline{u^i}=0, \, i=1,2$ as long as the solution is smooth. 

Now we study the 2D free energy of $n$ on $\mathbb{T}^2$:
\begin{align}\label{ET2}
E_{\mathbb{T}^2}[n,u]=\int_{\mathbb{T}^2} n\log n -\frac{1}{2}(n-\overline{n})c +\frac{1}{2}|u|^2dx.
\end{align}
\begin{lem}Consider the smooth solution to the equation \eqref{pePKS-NS-Torus}, the free energy $E_{\mathbb{T}^2}$ \eqref{ET2} is dissipated along the dynamics, i.e.,
\begin{align}\label{T3 E bound}
E_{\mathbb{T}^2}[n(t),u(t)] \leq E_{\mathbb{T}^2}[n_0,u_0],\quad \forall t\geq 0.
\end{align}
\end{lem}
\begin{proof}
Direct calculation of the time derivative of $\cF[n]$ can be estimated as follows
\begin{align}
\frac{d}{dt}E_{\mathbb{T}^2}[n,u]=&\int n_t(\log n-c)dx-\int|\na u|^2dx+\int u\cdot \na c ndx\\
=&-\int 
n(\na\log n-\na c )\cdot(\na \log n-\na c)dx+\int \na\cdot(u n)cdx-\int|\na u|^2dx+\int u\cdot \na c ndx\\
=&-\int n|\na \log n-\na c|^2dx-\int|\na u|^2dx\label{T3 d dt E}.
\end{align}
\end{proof}

The decaying free energy (\ref{T3 E bound}), together with a suitable logarithmic Hardy-Littlewood-Sobolev inequality yields a uniform-in-time bound on the positive component of the entropy $S^+[n]$. To explicitly derive the bound, we recall the following logarithmic Hardy-Littlewood-Sobolev inequality on a compact manifold:
\begin{thm}\cite{SW}
Let $\mathcal{M}$ be a two-dimensional, Riemannian, compact manifold. For all $M > 0$, there exists a constant $C(M)$ such that for all non-negative functions $f \in L^1(\mathcal{M})$ such that $f \log f \in L^1$, if $\int_{\mathcal{M}} f dx = M$, then
\begin{align}\label{log HLS T2}
\int_{\mathcal{M}}f \log fdx+\frac{2}{M}\iint_{\mathcal{M}\times\mathcal{M}} f(x) f(y) \log d(x,y) dxdy\geq -C(M),
\end{align}
where $d(x,y)$ is the distance on the Riemannian manifold.
\end{thm}
Since the logarithmic Hardy-Littlewood-Sobolev inequality \eqref{log HLS T2} is stated with respect to the distance on the torus, we cannot directly combine it with the decaying free energy \eqref{T3 E bound} here. To overcome this difficulty, we estimate the potential part of the free energy, i.e., $\frac{1}{2}\int(n-\overline{n})c dx$, from below. This is the main content of the next lemma.
\begin{lem}\label{lem: Free_Energy_T2}
There exists a constant $B>0$, such that the following estimate holds
\begin{align}
-\frac{1}{2}\int_{\mathbb{T}^2}(n-\overline{n})c dx= -\frac{1}{2}\int_{\mathbb{T}^2}(n-\overline{n})(-\de)^{-1}(n-\overline{n}) dx\geq\frac{1}{4\pi}\iint_{\mathbb{T}^2\times\mathbb{T}^2}\log d(y,x)n(y)n(x) dydx-BM^2.
\end{align}
\end{lem}
\begin{proof}
The proof of the lemma is the same as the parallel treatment in the paper \cite{BedrossianHe16}. For the sake of completeness, we provide the proof in the appendix.
\end{proof}
Combining Lemma \ref{lem: Free_Energy_T2} with (\ref{T3 E bound}) yields
\be\ba
 \cF&_{\mathbb{T}^2}[n_{0}, u_0]\\
 \geq& \left(1- \frac{M}{8\pi}\right)\int_{\mathbb{T}^2} n\log ndx+ \frac{M}{8\pi} \left( \int_{\mathbb{T}^2} n\log ndx +\frac{2}{M}\iint_{\mathbb{T}^2\times\mathbb{T}^2}n(y)\log d(y,x)n(x) dydx\right)+\frac{||u||_2^2}{2}-BM^2.
\ea\ee
Applying (\ref{log HLS T2}) in the above estimate, we obtain
\begin{align}
 \cF_{\mathbb{T}^2}[n_{0},u_0]\geq&\left(1-\frac{M}{8\pi}\right)\int_{\mathbb{T}^2} n\log ndx+\frac{||u||_2^2}{2}-C(M)-BM^2,
\end{align}
which results in
\begin{align}
\int_{\mathbb{T}^2} n\log ndx + \frac{1}{2(1-\frac{M}{8\pi})}||u||_2^2\leq& \frac{ \cF_{\mathbb{T}^2}[n_{0}, u_0]+C(M)+BM^2}{1-\frac{M}{8\pi}}.
\end{align}
Since the function $s\log s$ is bounded from below, the negative part of the entropy $S^-[n]=\int_{\mathbb{T}^2}n\log^- ndx$ is bounded on the torus. Therefore, there exists a constant $C_{L \log L}$ depending only on the initial data such that the following estimate holds:
\be
\int_{\mathbb{T}^2} n\log ^+n dx+||u||_2^2\leq C_{L\log L;L^2}(E_{\mathbb{T}^2}[n_{0}, u_0], M) < \infty.
\ee
The estimation above yields the following lemma.
\begin{lem} \label{lem:nlognbdT2}
If the total mass is bounded $||n_{0}||_{L^1}<8\pi$, there exists a constant $C_{L\log L}(n_{0},u_0)$ such that
\begin{align}\label{T3_entropy_bound} \int_{\mathbb{T}^2} n(t,x)\log^+ n(t,x)dx +||u(t )||_{L^2_x}^2\leq C_{L\log L;L^2}(E_{\mathbb{T}^2}[n_{0},u_0], M)<\infty,\quad\forall t\in[0,\infty) . 
\end{align}
\end{lem}

As in the plane case, the uniform in time bound on the positive part of the entropy $S^+[n]$ yields the bound on the $L^p$ norms. This is the content of the next lemma.
\begin{lem} \label{lem:L2T}Assume that the entropy is bounded in the sense that (\ref{T3_entropy_bound}) holds, then there exists a constant $C_{1,\infty}=C_{1,\infty}(n_0, u_0)$ such that the following estimate holds
\begin{align}\label{T3_L2_bound_of_n0}
||n(t)||_{
L^1\cap L^\infty}\leq C_{1,\infty}(||n_{0}||_{L^1\cap L^\infty},E_{\mathbb T^2}[n_0,u_0])<\infty,\quad \forall t\in[0,\infty).
\end{align}
\end{lem}
The proof is a small variation of classical Patlak-Keller-Segel techniques (see e.g. \cite{JagerLuckhaus92,BlanchetEJDE06}). Before presenting the proof, we recall the following Gagliardo-Nirenberg-Sobolev inequality on $\Torus^d$:
Suppose $v\in H^1(\Torus^d), d\geq 2$, and $\int v dx=0$ . Assume that $q,r>0,\infty>q>r$, and $\frac{1}{d}-\frac{1}{2}+\frac{1}{r}>0$. Then
\begin{align}\label{GNS Td}
||v||_{L^q}\leq C(d,q)||\na v||_{L^2}^{a}||v||_{L^r}^{1-a},\quad \int_{\Torus^d} v dx=0, \quad a=\frac{\frac{1}{r}-\frac{1}{q}}{\frac{1}{d}-\frac{1}{2}+\frac{1}{r}}.
\end{align}
For a fixed $d$, the constant $C(d,q)$ is bounded uniformly when $q$ varies in any compact set in $(0,\infty)$. 
\begin{proof}[Proof of Lemma \ref{lem:L2T}] We focus on the $L^2$ estimate. Let $K>\max\{1,\overline{n}\}$ be a constant, to be chosen later. 
Observe that \eqref{T3_entropy_bound} implies the following:
\begin{align}\label{T3_pick_K}
\int(n-K)_+dx\leq\int_{n>K}ndx\leq\frac{1}{\log(K)}\int_{n>K}n\log^+ (n)dx\leq \frac{C_{L\log L}}{\log(K)}.
\end{align}
Next, via \eqref{pePKS-NS-Torus} and the divergence-free property of the vector field $u$, there holds 
\begin{align}
\frac{1}{2}&\frac{d}{dt}\int (n-K)_+^2dx\nonumber\\
=&\int(n-K)_+[\de n-\na\cdot(n \na c)]dx\nonumber\\
=&-\int |\na((n-K)_+)|^2dx+\frac{1}{2}\int(n-K)^3_+dx+\frac{3K-\overline{n}}{2}\int(n-K)_+^2dx +{K^2-K\overline{n}}\int(n-K)_+dx\nonumber\\
\leq&-\frac{7}{8}\int |\na((n-K)_+)|^2dx+\frac{1}{2}\int(n-K)^3_+dx+\frac{3K-\overline{n}}{2}\int(n-K)_+^2dx +{(K^2-K\overline{n})M}.\label{T3_d_dt_n0-K_2}
\end{align}

We start with the second term in \eqref{T3_d_dt_n0-K_2}. Consider the average-zero function $(n-K)_+-\overline{(n-K)_+}$ on $\Torus^2$, and  we apply the Gagliardo-Nirenberg-Sobolev inequality \eqref{GNS Td}   to deduce: 
\begin{align}
\int |(n-K)_+|^3dx\leq &C\left(\int |(n-K)_+-\overline{(n-K)_+}|^3 dx+\overline{(n-K)_+}^3\right) \\
\leq& C_{GNS}\int|\na(n-K)_+|^2dx\int (n-K)_+dx+CM^3.
\end{align}
From \eqref{T3_pick_K}, we choose $K$ depending only on $C_{L \log L}$ such that
\begin{align}\label{T3_d_dt_n0-K_2:1}
-\frac{7}{8}\int |\na((n-K)_+)|^2dx+\frac{1}{2}\int(n-K)^3_+dx\leq -\frac{1}{2}\int |\na((n-K)_+)|^2dx+CM^3.
\end{align}
Plugging \eqref{T3_d_dt_n0-K_2:1} into \eqref{T3_d_dt_n0-K_2} yields the following for some universal constant $B >0$, 
\begin{align}
\frac{1}{2}\frac{d}{dt}\int (n-K)_+^2dx
\leq & -\frac{1}{2}\int |\na((n-K)_+)|^2dx+{K B}{}\int(n-K)_+^2dx+{B K^2M}.\label{T3_d_dt_n0-K_2:2.5}
\end{align}
Recalling the Gagliardo-Nirenberg-Sobolev inequality  \eqref{GNS Td}, for a general function $v$ on $\Torus^2$,  the following Nash inequality holds
\be
||v||_{L^2(\Torus^2)}^2\leq C_N ||\na v||_{L^2(\Torus^2)}||v||_{L^1(\Torus^2)}+ C_N||v||_{L^1(\Torus^2)}^2.
\ee
Applying the Nash inequality in the estimate \eqref{T3_d_dt_n0-K_2:2.5} yields
\begin{align}\label{T3_d_dt_n0-K_2:3}
\frac{1}{2}\frac{d}{dt}|| (n-K)_+||_2^2\leq -\frac{1}{2B}\frac{||(n-K)_+||_2^4}{C_NM^2}+\frac{3KB}{2}||(n-K)_+||_2^2  +BK^2M.
\end{align}
Further note that
\begin{align}
\norm{n}_{L^2} & \leq\norm{(n - K)_+}_{L^2}+\norm{\min \{n, K\}}_{L^2}\leq \norm{(n - K)_+}_{L^2} + K^{1/2} M^{1/2}.\label{ineq:layercake}&  
\end{align}
The inequality \eqref{T3_L2_bound_of_n0} hence follows. 
Same as in the proof of Theorem \ref{thm:S+n_condition}, we apply energy estimates to derive the $L^4$-bound on the density $n$, which in turn implies the chemical gradient $||\na c||_\infty$ estimate through Morrey's inequality and the Calderon-Zygmund inequality. Now application of the Moser-Alikakos iteration yields that 
\begin{align}
||n(t)||_{L^\infty}\leq C_{1,\infty}(||n_0||_{L^1\cap L^\infty}, E_{\mathbb{T}^2}[n_0,u_0])<\infty,\quad \forall t\in[0,\infty).
\end{align}
This concludes the proof of the theorem.
\end{proof}
Next, we prove the higher regularity estimates using \eqref{T3_L2_bound_of_n0}.
\begin{lem}\label{Lemma:T_3_H^1_control_of_n}Consider the solution to the equation \eqref{pePKS-NS-Torus}, the following $H^s,\,2\leq s\in \mathbb{N}$ estimates hold on $[0,\infty)$:
\begin{align}\label{H^s control of n}
||n(t)||_{H^s(\Torus^2)}+|| u(t)||_{H^s(\mathbb{T}^2)} \leq C_{{H}^s}(||n_0||_{H^s},||u_0||_{H^s},\CC(||n_0||_{L^1\cap L^\infty},E_{\mathbb{T}^2}[n_0,u_0]))<\infty,\,\forall t\in[0,\infty).
\end{align}
\end{lem}
\begin{proof}
Before proving the lemma, we collect the inequalities we are going to apply. 
The $L^4$ boundedness of the Riesz transform on $\Torus^d$ (see, e.g., \cite{SteinWeiss} Chapter VII section 3) yields that
\begin{align}
||\na^2 c||_2=||\na^2(-\de)(n-\overline{n})||_2\leq C||n-\overline{n}||_2;\\
||\na^2 c||_4=||\na^2(-\de)(n-\overline{n})||_4\leq C||n-\overline{n}||_4.\label{na2c}
\end{align} 
Combining the Morrey's inequality and the Calderon-Zygmund inequality yields that 
\begin{align}\label{nac}
||\na c||_{L^\infty(\mathbb{T}^2)}\leq&C \norm{n-\overline{n}}_{L^3(\Torus^2)}.
\end{align}

Now we estimate the time evolution of $\dot H^1$ norm of the velocity $u$ with the identity \eqref{B_identity},  Gagliardo-Nirenberg-Sobolev inequality, the chemical gradient estimates \eqref{nac}, \eqref{na2c}, and the $L^p,\, p\geq 1$ controls of the cell density $n $ \eqref{T3_L2_bound_of_n0} as follows:
\begin{align*}
\frac{1}{2}\frac{d}{dt}||\na u||_2^2\leq &- ||\na^2 u||_2^2+  ||\na^2 u||_2|| n||_2||\na c||_\infty \\
\leq&-\frac{||\na u||_2^4}{4C_{GNS}||u-\overline{u}||_2^2}+C||n||_{L^1\cap L^\infty}^4.
\end{align*}
Combining the estimates \eqref{T3_entropy_bound}, \eqref{T3_L2_bound_of_n0} yields that 
\begin{align}\label{na_u_L2_T2}
||\na u(t)||_{L^2}\leq C_{  u;H^1}(C_{L\log L;L^2}, \CC, ||\na u_0||_2), \quad\forall t\in [0,\infty)
\end{align}
Similarly, we can estimate the time evolution of the $\dot H^1$ norm of $n$ with estimates \eqref{T3_L2_bound_of_n0}, \eqref{nac}, \eqref{na_u_L2_T2} as follows:
\begin{align*}
\frac{1}{2}\frac{d}{dt}||\na n||_2^2\leq&-\frac{1}{2}{||\na ^2 n||_2^2}+||\na n||_4^2||\na u||_2+||\na^2 n||_2||\na n||_2||\na c||_\infty+||\na^2n||_2|| n||_4||\na^2c||_4\\
\leq&-\frac{1}{2}{||\na ^2 n||_2^2}+C||\na^2 n||_2||\na n||_2||\na u||_2+ ||\na^2 n||_2||\na n||_2||\na c||_\infty+C||\na^2n||_2|| n||_4|| n-\overline{n}||_4 \\
\leq&-\frac{1}{2}{||\na ^2 n||_2^2}+\frac{1}{4}||\na^2 n||_2^2+C\left(||\na u||_2^2+ ||\na c||_\infty^2+||n-\overline {n}||_2^2\right)||\na n||_2^2+C\overline{n}^4\\
\leq&-\frac{||\na n||_2^4}{4C_{GNS} ||n-\overline{n}||_2^2}+C\left(C_ {u;H^1}^2+\CC^2 \right)||\na n||_2^2+CM^4.
\end{align*}
Combining the solution to the above differential inequality and \eqref{T3_L2_bound_of_n0}, \eqref{nac}, \eqref{na_u_L2_T2}, we see that
\begin{align*}
||\na n(t)||_2^2+||\na u(t)||_2^2\leq&C_{H^1}(||n_0||_{H^1},||u_0||_{H^1},||n_0||_{L^1\cap L^\infty}, E_{\mathbb{T}^2}[n_0,u_0])<\infty,\quad \forall t\in[0,\infty).
\end{align*}
Further iterate this argument yields the $H^s,\, s\geq 3$ bound. This finishes the proof of Lemma \ref{Lemma:T_3_H^1_control_of_n}.
\end{proof}

\subsubsection*{Acknowledgments}
The authors would like to thank Alexander Kiselev for many helpful discussions and careful reading of the first version of the draft. The second author would also like to thank Jian-Guo Liu for referring us to the paper \cite{KozonoMiuraSugiyama} and helpful discussions. 
{The authors would also like to thank the anonymous referees for pointing out that the $8\pi$-mass threshold is actually sharp on $\rr^2$, which greatly improves the result.} Both of the authors got partial support from the NSF-DMS grant 1848790. The second author got partial support from NSF grant RNMS11-07444 (KINet).  

\appendix
\section{Appendix}
\begin{proof}[Proof of Theorem \ref{thm:local_well_posedness}] We prove the local \emph{a-priori estimates} of the $H^s, \, s\geq 3$ norms of the velocity field $u$ and the density $n$ assuming that the solution is smooth. These bounds can be justified through standard approximation procedure. Since the approximation step is classical, we  refer the readers to Chapter 6 and 7 of \cite{RobinsonRodrigoSadowski16} for further details. 

We first derive the $L^2$ estimate of the density $n$. Recall the equation of the density
\begin{equation}\label{eqn:density1}
    \partial_t n+u \cdot \nabla n +\nabla \cdot (\nabla c n)=\Delta n.
\end{equation}
We multiply the equation \eqref{eqn:density1} by $n$ and integrate to obtain:
\begin{equation}
    \frac{1}{2}\frac{d}{dt} \int n^2dx +\int u \cdot \nabla\left(\frac{n^2}{2}\right)dx+ \int n \nabla \cdot (\nabla c n)dx  =- \int |\nabla n |^2dx.\label{Lp_test} 
\end{equation} 
Since $u$ is divergence-free, the second term on the left hand side of \eqref{Lp_test} is zero. For the third term on the left hand side of \eqref{Lp_test}, direct integration by parts yields that 
\begin{align*}
\begin{split}
    \int n  \nabla \cdot (\nabla c n)dx 
    =-\int n^{3}dx+\int \nabla c\cdot \nabla\left(\frac{n^2}{2}\right)dx=-\frac{1}{2}\int n^{3}dx.
\end{split}
\end{align*}
Combining this equation with \eqref{Lp_test}, and applying Gagliardo-Nirenberg-Sobolev inequality $||f||_3\leq C_{GNS}||f||^{\frac{2}{3}}_2 ||\nabla f||^{\frac{1}{3}}_2$ with $f=n $ and Young's inequality yield that
\begin{align}
\begin{split}\label{eqn:diff_inequ_n}
\frac{d}{dt}||n||_2^2 &\leq  - 2\int |\nabla n |^2dx+\int n^{3}dx \leq- 2||\nabla n ||^{2}_2+ C_{GNS}||\nabla n || _2  ||n||_2^2 \\
    &\leq- \frac{3}{2} ||\nabla  n ||_2^2 +C_{GNS}||n||_2^{ {4} }\leq C_{GNS}||n||_2^{4}.
\end{split}
\end{align}
By ordinary differential equation theory, we obtain that  there exists a small constant $\epsilon_2=\ep_2(C_{GNS},||n_{0}||_2)$ such that for time $t$ smaller than $\ep_2$, i.e., $0\leq t<\epsilon_2$, the $L^2$ norm of the solution $n$ is bounded:
\begin{equation}\label{eqn:local_control_n}
||n(t)||_2^2\leq2||n_0||_2^2,\quad \forall t\in[0,\ep_2].
\end{equation}
Once the $L^2$ bound of the density is achieved, we can estimate the $L^4$ norm of the chemical gradient $\na c$ on a short time interval $t\in[0, \ep_2]$. Applying the Hardy-Littlewood-Sobolev inequality, H\"older inequality and Young inequality, we estimate the chemical gradient $ \nabla c $ as follows 
\begin{align}
||\na c(t)||_{L^4}\leq C_{HLS}||n(t)||_{L^{4/3}}\leq C_{HLS}(M+||n(t)||_{L^2})\leq C(M, ||n_0||_2),\quad \forall t\in[0,\ep_2].    \label{chemical_gradient_estimate}
\end{align} 

Next we estimate the $L^2$ norm of the fluid velocity fields $u$. Recall the fluid equation after we apply the Leray projection operator $\mathbb{P}$:
\begin{equation}\label{eqn:flow_eqn}
    \partial_t u+\mathbb{P}((u\cdot \nabla)u)=-\mathbb{P}(-\Delta)u+\mathbb{P}(n\nabla c).
\end{equation}
Since the Leray projection is self-adjoint and the vector field $u$ is divergence free, multiplying \eqref{eqn:flow_eqn} by $u$ and integrating yields the following equality
\begin{equation}
\frac{d}{dt} \frac{1}{2}\int |u|^2dx +\int u\cdot ((u\cdot \nabla)u)dx=\int u  \Delta u dx+\int u \cdot(n\nabla c)dx.\label{L2_test}
\end{equation}
Due to the divergence-free property of the vector field $ u$, we have that the second term on the left hand side of \eqref{L2_test} vanishes. 
Combining the estimates \eqref{eqn:local_control_n}, \eqref{chemical_gradient_estimate}, the H\"older inequality and the Gagliardo-Nirenberg-Sobolev inequality yields  that the second term on the right hand side of \eqref{L2_test} is bounded on the time interval $[0,\ep_2]$,  
\begin{equation*}
\int u(t,x)\cdot (n(t,x)\nabla c(t,x))dx 
\leq ||\nabla c(t)||_4 ||u(t)||_4||n(t)||_2\leq C(M ,||n_0||_{L^2 }) ||u(t)||_2^{1/2}||\na u(t)||_2^{1/2},\quad \forall t\in[0,\ep_2].
\end{equation*}
Combining these estimates with \eqref{L2_test}, we apply the Young's inequality to obtain that 
\begin{equation}\label{eqn:diff_inequ_u}
    \frac{d}{dt}||u(t)||_2^2\leq -\frac{1}{2}||\na u||_2^2+ C(M, ||n_0||_{L^2 }) ||u(t)||_2^{2/3},\quad \forall t\in[0,\ep_2].
\end{equation}
Therefore, we have the following local control over $||u(t)||_2$
 \begin{align} 
 ||u(t)||_2\leq C(M,||n_0||_{L^2 }, ||u_0||_2)<\infty,\quad \forall t\in [0,\epsilon_2].
 \end{align} 
Next we estimate the $H^1$-norm. Before estimating the $\dot H^1$ norms, we recall estimates \eqref{na2c_R2}, 
\begin{align}
||\na^2 c||_2=||\na^2(-\de)n||_2\leq C ||n||_2,\quad
||\na^2 c||_4=||\na^2(-\de)n||_4\leq C||n||_4.\label{na2c_R2_appendix}
\end{align} Now we estimate the time evolution of the $\dot H^1$ seminorm of the velocity $u$ \eqref{eqn:flow} with the divergence-free condition of $u$, the self-adjoint property of $\mathbb{P}$,  the Gagliardo-Nirenberg-Sobolev inequality,  and the chemical gradient estimates \eqref{chemical_gradient_estimate}, \eqref{na2c_R2_appendix} as follows:
\begin{align*}
\frac{1}{2}\frac{d}{dt}\sum_{j=1}^2||\pa_j u||_2^2=&-\sum_{j=1}^2\int |\na \pa_j u|^2 dx-\sum_{j=1}^2\int \pa_j B(u,u) \cdot \pa_j u dx+\sum_{j=1}^2\int \pa_j \mathbb{P}(\na c n)\cdot \pa_j u dx\\
\leq &-\frac{1}{2}||\na^2 u||_2^2+C\left(||\na u||_2^4+||  n||_4^2||\na c||_4^2 \right)\\ 
\leq&C||\na u||_2^4 +C||\na n||_2||n||_2(M^2+ ||n||_2^2).
\end{align*}
Similarly, we estimate the time evolution of the $\dot H^1$ seminorm of $n$ using the divergence free property of $u$,  the Gagliardo-Nirenberg-Sobolev inequality, and  the chemical gradient estimate \eqref{chemical_gradient_estimate}, \eqref{na2c_R2_appendix}  as follows:
\begin{align*}
\frac{1}{2}\frac{d}{dt}||\na n||_2^2\leq&-\frac{1}{2}{||\na ^2 n||_2^2}+||\na n||_4^2||\na u||_2+||\na^2 n||_2||\na n||_4||\na c||_4+||\na^2n||_2|| n||_4||\na^2c||_4\\
\leq&-\frac{1}{2}{||\na ^2 n||_2^2}+C||\na^2 n||_2||\na n||_2||\na u||_2+C||\na^2 n||_2^{3/2}||\na n||_2^{1/2}(M+||n ||_2)+C||\na^2n||_2|| n||_2||\na n||_2 \\
\leq&-\frac{1}{4}||\na^2 n||_2^2+C||\na n||_2^4+C||\na u||_2^4+C (1+M+||n||_2)^4  ||\na n||_2^2.
\end{align*}Combining these estimations on time evolution of $||n||_{\dot H^1}^2$ and $ ||u||_{\dot H^1}^2$
with the $L^2$ bound on the cell density $n$ \eqref{C2infty} and the assumption on the fluid velocity $u$ \eqref{Energy_u_condition} yields that there exists a universal constant $C$ such that 
\begin{align*}
\frac{1}{2}\frac{d}{dt}
(||\na n||_2^2+||\na u||_2^2)\leq&C||\na u||_2^4+C||\na n||_2^4 + C(M,||n||_2) ||\na n||_2^2+C (|n||_2,M).
\end{align*}
Now by standard ODE theory and the $L^2$-bound \eqref{eqn:local_control_n}, we obtain that
\begin{align}\label{H_1_bound_appendix}
||\na n(t)||_2+||\na u(t)||_2\leq&C_{H^1}(
 ||n_0||_{L^1},  ||n_0||_{H^1},||u_0||_{H^1})<\infty,\quad \forall t\in[0,\ep],
 \end{align} 
for some small enough $\ep=\ep(||n_0||_{L^1},
 ||n_0||_{H^1},||u_0||_{H^1})$.

Now we can apply the similar procedure showed in the proof of Theorem \ref{thm:S+n_condition} to gain control over $H^s$ norms of $u$ and $n$ on the interval $[0,\ep]$. Following the arguments in Chapter 7 of \cite{RobinsonRodrigoSadowski16}, higher space-time regularity of the solutions can be obtained.  This concludes the proof of the theorem. 
\end{proof}

\begin{proof}[Proof of Lemma \ref{lem: Free_Energy_T2}]
Let $x \in \Torus^2$ be fixed. Define the cut-off function $\varphi_x(y)\in C^\infty$ such that
\be\ba
supp(\varphi_x)=&B(x,1/4),\\
\varphi_x(y)\equiv& 1,\forall y\in B(x,1/8),\\
supp(\na \varphi_x(y))\subset& \overline{B}(x,1/4)\backslash B(x,1/8).
\ea\ee
By extending $n(y)$ and $c(y)$ periodically to $\rr^2$, we can rewrite the equation $-\de c=n-\overline{n}$ on $\mathbb{T}^2$ such that it is posed on $\rr^2$:
\be
-\de_y (\varphi_x(y)c(y))=(n(y)-\overline{n})\varphi_x(y)-2\na_y \varphi_x(y)\cdot\na_y c(y)-\de_y\varphi_x(y)c(y).
\ee
Using the fundamental solution of the Laplacian on $\rr^2$:
\be\ba
c(x)=&c(x)\varphi_x(x)\\
=&-\frac{1}{2\pi}\int_{\rr^2}\log|x-y|\bigg((n(y)-\overline{n})\varphi_x(y)-2\na_y \varphi_x(y)\cdot\na_y c(y)-\de_y\varphi_x(y)c(y)\bigg)dy\\
=&-\frac{1}{2\pi}\int_{|x-y|\leq \frac{1}{4}}\log|x-y|(n(y)-\overline{n})\varphi_x(y)dy-\frac{1}{\pi}\int_{|x-y|\leq \frac{1}{4}}\na_y\cdot(\log|x-y|\na_y \varphi_x(y)) c(y)dy\\
&+\frac{1}{2\pi}\int_{|x-y|\leq \frac{1}{4}}\log|x-y|\de_y\varphi_x(y)c(y)dy.
\ea\ee
Due to the support of $\varphi_x$, we can identify the above with an analogous integral on $\mathbb{T}^2$ with $\abs{x-y}$ replaced by $d(x,y)$.
Therefore, we have the following estimate on the interaction energy,
\begin{align*}
-&\frac{1}{2}\int\displaylimits_{\mathbb{T}^2}(n(x)-\overline{n})c(x) dx\\
=&\frac{1}{4\pi}\iint\displaylimits_{\substack{\mathbb{T}^2\times\mathbb{T}^2\\d(x,y)\leq \frac{1}{4}}}\log d(x,y)(n(x)-\overline{n})(n(y)-\overline{n})\varphi_x(y)dydx+\frac{1}{2\pi}\iint\displaylimits_{\substack{\mathbb{T}^2\times\mathbb{T}^2\\ \frac{1}{8}\leq d(x,y)\leq \frac{1}{4}}}(n(x)-\overline{n})\na_y\cdot(\log d(x,y)\na_y \varphi_x(y)) c(y)dydx\\
&-\frac{1}{4\pi}\iint\displaylimits_{\substack{\mathbb{T}^2\times\mathbb{T}^2\\\frac{1}{8}\leq d(x,y)\leq \frac{1}{4}}}(n(x)-\overline{n})\log d(x,y)\de_y\varphi_x(y)c(y)dydx\\
=&\frac{1}{4\pi}\iint\displaylimits_{d(x,y)\leq \frac{1}{8}}\log d(x,y)(n(x)-\overline{n})(n(y)-\overline{n})dydx+\frac{1}{4\pi}\iint\displaylimits_{\frac{1}{8}\leq d(x,y)\leq \frac{1}{4}}\log d(x,y)(n(x)-\overline{n})(n(y)-\overline{n})\varphi_x(y)dydx\\
&+\frac{1}{2\pi}\iint\displaylimits_{\frac{1}{8}\leq d(x,y)\leq \frac{1}{4}}(n(x)-\overline{n})\na_y\cdot(\log d(x,y)\na_y \varphi_x(y)) c(y)dydx-\frac{1}{4\pi}\iint\displaylimits_{\frac{1}{8}\leq d(x,y)\leq \frac{1}{4}}(n(x)-\overline{n})\log d(x,y)\de_y\varphi_x(y)c(y)dydx\\
=&\frac{1}{4\pi}\iint\displaylimits_{\mathbb{T}^2\times\mathbb{T}^2}\log d(x,y)n(x)n(y)dydx-\frac{1}{4\pi}\iint\displaylimits_{d(x,y)> \frac{1}{8}}\log d(x,y)n(x)n(y)dydx\\
&-\frac{1}{2\pi}\overline{n}\iint\displaylimits_{d(x,y)\leq \frac{1}{8}}\log d(x,y)n(x)dydx+\frac{1}{4\pi}\overline{n}^2\iint\displaylimits_{d(x,y)\leq \frac{1}{8}}\log d(x,y)dydx
\end{align*}
\begin{align*}
&+\frac{1}{4\pi}\iint\displaylimits_{\frac{1}{8}\leq d(x,y)\leq \frac{1}{4}}\log d(x,y)(n(x)-\overline{n})(n(y)-\overline{n})\varphi_x(y)dydx\\
&+\frac{1}{2\pi}\iint\displaylimits_{\frac{1}{8}\leq d(x,y)\leq \frac{1}{4}}(n(x)-\overline{n})\na_y\cdot(\log d(x,y)\na_y \varphi_x(y)) c(y)dydx-\frac{1}{4\pi}\iint\displaylimits_{\frac{1}{8}\leq d(x,y)\leq \frac{1}{4}}(n(x)-\overline{n})\log d(x,y)\de_y\varphi_x(y)c(y)dydx.
\end{align*}
The 2nd, 3rd, 4th, 5th terms in the last line are bounded below by $-BM^2$ for some constant $B > 0$.
The 6th and 7th terms are bounded below by $-B M ||c||_{L^1}$ for some constant $B > 0$, using the fact that $\na_y\cdot(\log|x-y|\na_y \varphi_x(y))$ and $\log|x-y|\de_y\varphi_x(y)$ are bounded in the region $\frac{1}{8}\leq|x-y|\leq \frac{1}{4}$.
Denoting $K(y)$ to be the fundamental solution of the Laplacian on $\mathbb{T}^2$, by Young's inequality, we have
\be\ba
||c||_{L^1(\mathbb{T}^2)}=||K\ast (n-\overline{n})||_{L^{1}(\mathbb{T}^2)}\leq||K||_{L^1(\mathbb{T}^2)}||n-\overline{n}||_{L^{1}(\mathbb{T}^2)}\leq BM.
\ea\ee
\end{proof}
\def\cprime{$'$}

\end{document}